\documentclass{amsart}
\usepackage[english]{babel}
\usepackage{amsmath,amssymb,amsthm}
\usepackage{stmaryrd}
\usepackage{bbm}
\usepackage{color}

\newtheorem{theorem}{Theorem}[section]
\newtheorem{claim}[theorem]{Proposition}
\newtheorem{remark}[theorem]{Remark}

\newtheorem{corollary}[theorem]{Corollary}
\newtheorem{quest}[theorem]{Question}

\begin{document}
\newcounter{gl}
\title{Derivations on symmetric quasi-Banach ideals of compact operators}
\author{A. F. Ber}
\address{Department of Mathematics,National University of Uzbekistan,
Vuzgorodok, 100174, Tashkent, Uzbekistan}
\email{ber@ucd.uz}

\author{V. I. Chilin }
\address{Department of Mathematics, National University of Uzbekistan,
Vuzgorodok, 100174, Tashkent, Uzbekistan}
\email{chilin@ucd.uz}

\author{G. B. Levitina }
\address{Department of Mathematics, National University of Uzbekistan,
Vuzgorodok, 100174, Tashkent, Uzbekistan}
\email{bob\_galina@mail.ru}

\author{F. A. Sukochev}
\address{School of Mathematics and Statistics, University of New South Wales, Sydney, NSW 2052, Australia }
\email{f.sukochev@unsw.edu.au}
\date{}
\begin{abstract}
Let $\mathcal{I,J}$ be symmetric quasi-Banach ideals of compact operators on an infinite-dimensional complex Hilbert space $H$, let $\mathcal{J:I}$ be a space of multipliers from $\mathcal{I}$ to $\mathcal{J}$. Obviously, ideals $\mathcal{I}$ and $\mathcal{J}$ are quasi-Banach algebras and it is clear that ideal $\mathcal{J}$ is a  bimodule for $\mathcal{I}$. We study the set of all derivations from $\mathcal{I}$ into $\mathcal{J}$. We show that any such derivation is automatically continuous and there exists an operator $a\in\mathcal{J:I}$ such that $\delta(\cdot)=[a,\cdot]$, moreover $\|a\|_{\mathcal{B}(H)}\leq\|\delta\|_\mathcal{I\to J}\leq 2C\|a\|_\mathcal{J:I}$, where $C$ is the modulus of concavity of the quasi-norm $\|\cdot\|_\mathcal{J}$. In the special case, when $\mathcal{I=J=K}(H)$ is a symmetric Banach ideal of compact operators on $H$ our result yields the classical fact that any derivation $\delta$ on $\mathcal{K}(H)$ may be written as $\delta(\cdot)=[a,\cdot]$, where $a$ is some bounded operator on $H$ and $\|a\|_{\mathcal{B}(H)}\leq\|\delta\|_\mathcal{I\to I}\leq 2\|a\|_{\mathcal{B}(H)}$.
\end{abstract}
\maketitle

\bigskip

\section{Introduction}
\large
Let $\mathcal{I,J}$ be ideals of compact operators on $H$. Obviously, $\mathcal{J}$ is an $\mathcal{I}$-module and we can consider the set ${\rm Der}(\mathcal{I,J})$ of all derivations $\delta\colon \mathcal{I}\to \mathcal{J}$. Consider two closely related questions (here, $\mathcal{B}(H)$ is the set of all bounded linear operators on $H$):

\begin{quest}\label{q1} Let $\delta\in {\rm Der}(\mathcal{I,J})$. Does there exist a bounded operator $a\in\mathcal{B}(H)$ such that $\delta(x)=[a,x]$ for every $x\in \mathcal{I}$?
\end{quest}

\begin{quest}\label{q2}  What is the set $D(\mathcal{I,J})=\{a\in\mathcal{B}(H): \quad [a,x]\in\mathcal{J},\ \forall x\in\mathcal{I}\}$?
\end{quest}

The second question was completely answered by M. J. Hoffman in \cite{Hoffman}, who also coined the term $\mathcal{J}$-essential commutant of $\mathcal{I}$ for the set $D(\mathcal{I,J})$. We completely answer the first question in the setting when the ideals $\mathcal{I,J}$ are symmetric quasi-Banach (see precise definition in the next section). In this setting, it is also natural to ask

\begin{quest}\label{q3}
Let $\delta\in {\rm Der}(\mathcal{I,J})$. Is it continuous?
\end{quest}

Of course, if $\delta\in {\rm Der}(\mathcal{I,J})$ is such that $\delta(x)=[a,x]$  for some $a\in\mathcal{B}(H)$ (that is when $\delta$ is implemented by the operator $a$), then $\delta$ is a continuous mapping from $(\mathcal{I},\|\cdot\|_{\mathcal{I}})$ to $(\mathcal{J},\|\cdot\|_{\mathcal{J}})$, that is a positive answer to Question \ref{q1} implies also a positive answer to Question \ref{q3}. However, in this paper, we are establishing a positive answer to Question \ref{q1} via firstly answering Question \ref{q3} in positive. Both these results (Theorem \ref{th6} and Theorem \ref{th7}) are proven in Section 3. We also provide a detailed discussion of the $\mathcal{J}$-essential commutant of $\mathcal{I}$ in Section 4.

It is also instructive to outline a connection between Questions \ref{q1} and \ref{q3} with some classical results.  It is well known (\cite{Sak}, Lemma 4.1.3) that every derivation on a $C^*$-algebra is norm continuous. In fact, this also easily follows from the following well-known fact (\cite{Sak}, Corollary 4.1.7) that every derivation on a $C^*$-algebra $\mathcal{M}\subset\mathcal{B}(H)$ is given by a reduction of an inner derivation on a von Neumann algebra $\overline{\mathcal{M}}^{wo}$ (the weak closure of $\mathcal{M}$ in the $C^*$-algebra $\mathcal{B}(H)$). The latter result (\cite{Sak}, Lemma 4.1.4 and Theorem 4.1.6), in the setting when $\mathcal{M}$ is a $C^*$-algebra $\mathcal{K}(H)$  of all compact operators on $H$ states that for every derivation $\delta$ on $\mathcal{M}$ there exists an operator $a\in\mathcal{B}(H)$ such that $\delta(x)=[a,x]$ for every $x\in\mathcal{K}(H)$, in addition, $\|a\|_{\mathcal{B}(H)}\leqslant\|\delta\|_{\mathcal{M}\to\mathcal{M}}$. The ideal $\mathcal{K}(H)$ equipped with the uniform norm is an element from the class of so-called symmetric Banach operator ideals in $\mathcal{B}(H)$ and evidently this example also suggests the statements of Questions \ref{q1} and \ref{q3}. In the case of Schatten ideals $C_p(H)=\{x\in\mathcal{K}(H):\|x\|_p=\mathrm{tr}(|x|^p)^{\frac{1}{p}}<\infty\}$, where $|x|=(x^*x)^{\frac{1}{2}},1\leqslant p<\infty$,  somewhat similar problems concerning derivations from $C_p(H)$ to $C_r(H)$ were also considered in the work by E.Kissin and V.S.Shulman \cite{K-Sh}. In particular,  it is shown in \cite{K-Sh} that every closed $*$-derivation $\delta$ from $C_p(H)$ to $C_r(H)$ is implemented by a symmetric operator $S$, in addition the domain  $D(\delta)$ of $\delta$ is dense $*$-subalgebra in $C_p(H)$. In our case, we have $D(\delta)=C_p$ and it follows from our results that the derivation $\delta$ is necessarily continuous and implemented by an operator $a\in\mathcal{B}(H)$.

It is also worth to mention that Hoffman's results in \cite{Hoffman} were an extension of earlier results by J.W.Calkin \cite{Calkin} who considered the case when $\mathcal{I}=\mathcal{B}(H)$. Recently, Calkin's and Hoffman's results were extended to the setting of general von Neumann algebras in  \cite{B-S} and, in the special setting when $\mathcal{I}=\mathcal{J}$, Questions \ref{q1} and \ref{q3} were also discussed in \cite{B-S_d}. However, our methods in this paper are quite different from all the approaches applied in \cite{B-S, Calkin, Hoffman, K-Sh}.

As a corollary of solving Questions \ref{q1} and \ref{q3}, in Theorem \ref{pr10} we present a description of all derivations $\delta$ acting from a symmetric quasi-Banach ideal $\mathcal{I}$ into a symmetric quasi-Banach ideal $\mathcal{J}$. Indeed, every such derivation $\delta$ is an inner derivation $\delta(\cdot)=\delta_a(\cdot)=[a,\cdot]$, where $a$ is some operator from $\mathcal{J}$-dual space $\mathcal{J:I}$ of $\mathcal{I}$. Recall that $D(\mathcal{I,J})=\mathcal{J:I}+\mathbb{C}\mathbbm{1}$ \cite{Hoffman}, where $\mathbbm{1}$ is the identity operator in $\mathcal{B}(H)$. Theorem \ref{pr10} gives a complete answer to Question \ref{q2}. In particular, using the equality  $C_r:C_p=C_q, 0<r<p<\infty, \frac{1}{q}=\frac{1}{r}-\frac{1}{p}$ we recover Hoffman's result that any derivation $\delta\colon C_p\to C_r$ has a form $\delta=\delta_a$ for some $a\in C_q$. If $0<p\leqslant r<\infty$, then $D(C_p,C_r)=\mathcal{B}(H)$.

When $\mathcal{I,J}$ are arbitrary symmetric quasi-Banach ideals of compact operators and $\mathcal{I}\subseteq\mathcal{J}$, then $\mathcal{J:I}=\mathcal{B}(H)$, and, in this case, a linear operator $\delta\colon\mathcal{I}\to\mathcal{J}$ is a derivation if and only if  $\delta=\delta_a$ for some $a\in\mathcal{B}(H)$. However, if $\mathcal{I}\nsubseteq\mathcal{J}$, then to obtain a complete description of $\mathcal{J}$-essential commutant of $\mathcal{I}$  we need a procedure of finding $\mathcal{J:I}$.

To this end, we use the classical Calkin's correspondence  between two-sided ideals $\mathcal{I}$ of compact operators and rearrangement invariant solid sequence subspaces $E_\mathcal{I}$ of the space $c_0$ of null sequences. The meaning of this correspondence is the following. Take a compact operator $x\in\mathcal{I}$ and consider a sequence of eigenvalues $\{\lambda_n(x)\}_{n=1}^\infty\in c_0$. For each sequence $\xi=\{\xi_n\}\in c_0$, let $\xi^*=\{\xi_n^*\}_{n=1}^\infty$ denote a decreasing rearrangement of the sequence $|\xi|=\{|\xi_n|\}_{n=1}^\infty$. The set
$$E_\mathcal{I}:=\{\{\xi_n\}_{n=1}^\infty\in c_0:\{\xi_n^*\}_{n=1}^\infty=\{\lambda_n^*(|x|)\}_{n=1}^\infty \mbox{ for some } x\in\mathcal{I}\},$$
is a  solid linear subspace in the Banach lattice $c_0$. In addition, the space $E_\mathcal{I}$ is rearrangement invariant, that is if $\eta\in c_0,\xi\in E_\mathcal{I},\eta^*=\xi^*$, then $\eta\in E_\mathcal{I}$. Conversely, if $E$ is a rearrangement invariant solid sequence subspace in $c_0$, then
$$C_E=\{x\in\mathcal{K}(H):\{\lambda_n(|x|)\}_{n=1}^\infty\in E\}$$
is a two-sided ideal of compact operators from $\mathcal{B}(H)$.

For the proof of the following theorem we refer to Calkin's original paper, \cite{Calkin}, and to B.~Simon's book, (\cite{Simon},Theorem 2.5).

\begin{theorem}
The correspondence $\mathcal{I} \leftrightarrow E_\mathcal{I}$ is a bijection between rearrangement invariant solid spaces in $c_0$ and two-sided ideals of compact operators.
\end{theorem}
In the recent paper \cite{K-S} this correspondence has been extended to symmetric quasi-Banach (Banach) ideals and $p$-convex symmetric quasi-Banach (Banach) sequence spaces.
We use the notation $\|\cdot\|_{\mathcal{B}(H)}$ and $\|\cdot\|_\infty$ to denote the uniform norm on $\mathcal{B}(H)$ and on $l_\infty$ respectively.

Recall, that a two-sided ideal $\mathcal{I}$ of compact operators from $B(H)$ is said to be symmetric quasi-Banach (Banach) ideal if it is equipped with a complete quasi-norm (respectively, norm) $\|\cdot\|_{\mathcal{I}}$ such that
$$\|axb\|_{\mathcal{I}}\leqslant\|a\|_{\mathcal{B}(H)}\|x\|_{\mathcal{I}}\|b\|_{\mathcal{B}(H)},\quad x\in\mathcal{I},a,b\in \mathcal{B}(H).$$
A symmetric sequence space $E\subset c_0$ is a rearrangement invariant solid sequence space equipped with a complete quasi-norm (respectively, norm) $\|\cdot\|_E$ such that
$\|\eta\|_E\leq \|\xi\|_E$ for every $\xi\in E$ and $\eta\in c_0$ such that $\eta^*\leqslant\xi^*$.

It is clear that if $(\mathcal{I}, \|\cdot\|_{\mathcal{I}})$ is a symmetric quasi-Banach ideal of compact operators, $x\in\mathcal{I}$ and  $y\in \mathcal{B}(H)$ is such that $\{\lambda_n^*(|y|)\}_{n=1}^\infty\leqslant\{\lambda_n^*(|x|)\}_{n=1}^\infty$, then $y\in\mathcal{I}$ and $\|y\|_{\mathcal{I}}\leq\|x\|_{\mathcal{I}}$. In Theorem \ref{inverse K-S} we show that if $E_\mathcal{I}$ is a symmetric space in $c_0$ corresponding to symmetric quasi-Banach ideal $\mathcal{I}$, then setting $\|\xi\|_{E_\mathcal{I}}:=\|x\|_\mathcal{I}$ (where $x\in \mathcal{I}$ is such that $\xi^*=\{\lambda_n^*(|x|)\}_{n=1}^\infty$) we obtain that $(E_\mathcal{I}, \|\cdot\|_{E_\mathcal{I}})$ is a symmetric quasi-Banach sequence space. The converse implication is much harder \cite{K-S}.
\begin{theorem}\label{K_S}
If $(E,\|\cdot\|_E)$ is a symmetric Banach (respectively, $p$-convex symmetric quasi-Banach) sequence space in $c_0$, then $C_E$ equipped with the norm
$$\|x\|_{C_E}:=\|\{\lambda_n^*(|x|)\}_{n=1}^\infty\|_E$$
is a symmetric Banach (respectively, $p$-convex quasi-Banach) ideal of compact operators from $\mathcal{B}(H)$.
\end{theorem}

In \cite{Gar} it was shown that for $\mathcal{J}=C_1$  is the
trace class and an arbitrary two-sided ideal $\mathcal{I}$ with
$C_1\subset\mathcal{I}\subset\mathcal{K}(H)$ the $C_1$-dual space
(also sometimes called the K\" othe dual) $\mathcal{I}^\times:=
C_1:\mathcal{I}$ of $\mathcal{I}$ is precisely an ideal
corresponding to symmetric sequence space $l_1:E_\mathcal{I}$,
where $l_1:E_\mathcal{I}$ is $l_1$-dual space of $E_\mathcal{I}$
(see precise definitions in section 4). If $\mathcal{I}$ is a
symmetric Banach ideal of compact operators, then $C_1$-dual space
$\mathcal{I}^\times$ is symmetric Banach ideal of compact operator
and norms on $C_1:\mathcal{I}$ and $C_{l_1:E_\mathcal{I}}$ are
equal \cite{DDdP}. We extend these results to arbitrary symmetric
quasi-Banach ideals $\mathcal{I,J}$ of compact operators with
$\mathcal{I}\nsubseteq\mathcal{J}$, that allows to describe
completely all derivations from one symmetric quasi-Banach ideal
to another. In addition, we use the technique of $\mathcal{J}$-dual spaces in order to obtain the estimation $\|\delta_a\|_{\mathcal{I}\to\mathcal{J}}\leqslant 2\|a\|_\mathcal{J:I}$ for an arbitrary derivation $\delta=\delta_a\colon\mathcal{I}\to\mathcal{J},a\in\mathcal{J:I}$. This result extends a well-known estimation $\|\delta_a\|_{\mathcal{M}\to\mathcal{M}}\leqslant 2\|a\|_{\mathcal{B}(H)},a\in\mathcal{B}(H)$ obtained by L.Zsido \cite{Zsido} for a derivation $\delta_a$ acting in an arbitrary von Neumann algebra $\mathcal{M}\subset\mathcal{B}(H)$.

\section{Preliminaries}
Let $H$ be an infinite-dimensional Hilbert space over the field $\mathbb{C}$ of complex numbers and $\mathcal{B}(H)$ be the $C^*$-algebra of all bounded linear operators on $H$.
Set
\begin{gather*}\mathcal{B}_h(H)=\{x\in \mathcal{B}(H): x^*=x\},
\\
\mathcal{B}_+(H)=\{x\in \mathcal{B}_h(H): \forall \varphi\in H \quad(x(\varphi),\varphi)\geqslant 0\},
\\
\mathcal{P}(H)=\{p\in\mathcal{B}(H):p=p^2=p^*\}.
\end{gather*}
It is well known (\cite{K-R1}, Ch. 2, \S 4) that
$\mathcal{B}_+(H)$ is a proper cone in  $\mathcal{B}_h(H)$ and
with the partial order given by $x\leqslant y \Leftrightarrow
y-x\in \mathcal{B}_+(H)$ the set $\mathcal{B}_h(H)$ is a partially
ordered vector space over the field $\mathbb{R}$ of real numbers,
satisfying $y^*xy\geqslant 0$ for all $y\in \mathcal{B}(H), x\in
\mathcal{B}_+(H)$. Note, that
$-\|x\|_{\mathcal{B}(H)}\mathbbm{1}\leqslant x\leqslant
\|x\|_{\mathcal{B}(H)}\mathbbm{1}$ for all $x\in
\mathcal{B}_h(H)$. It is known (see e.g. \cite{K-R1}, Ch.4, \S 2,
Proposition 4.2.3) that every operator $x$ in $\mathcal{B}_h(H)$
can be uniquely written as follows: $x=x_+-x_-$, where
$x_+,x_-\in\mathcal{B}_+(H)$ and $x_+x_-=0$. In addition, every
operator $x\in\mathcal{B}(H)$ can be represented as $x=u|x|$,
where $|x|=(x^*x)^{\frac{1}{2}}$ and $u$ is a partial isometry in
$\mathcal{B}(H)$ such that $u^*u$ is the right support of $x$
(\cite{R-S}, Ch. \setcounter{gl}{6}\Roman{gl}, \S 5, Theorem
\setcounter{gl}{6}\Roman{gl}.10).

We need the following useful proposition.

\begin{claim}(\cite{M-Ch}, Ch. 2, \S 4, Proposition 2.4.3)\label{pr} If
 $x,y\in\mathcal{B}_+(H), x\leqslant y$, then there exists an operator $a\in\mathcal{B}(H)$ such that $\|a\|_{\mathcal{B}(H)}\leqslant 1$ and $x=a^*ya$.
\end{claim}

Let $\mathcal{K}(H)$ be a two-sided ideal in $\mathcal{B}(H)$ of all compact operators and $x\in\mathcal{K}(H)$. The eigenvalues $\{\lambda_n(|x|)\}_{n=1}^\infty$ of the operator $|x|$ arranged in decreasing order and repeated according to algebraic multiplicity are called singular values of the operator $x$, i.e. $s_n(x)=\lambda_n(|x|), n\in\mathbb{N}$, where $\lambda_1(|x|)\geqslant\lambda_2(|x|)\geqslant\dots$ and $\mathbb{N}$ is the set of all natural numbers. We need the following properties of singular values.

\begin{claim}\label{s-chisla1} (\cite{G-K}, Ch.\setcounter{gl}{2}\Roman{gl} )
a) $s_n(x)=s_n(x^*), s_n(\alpha x)=|\alpha|s_n(x)$ for all $x\in\mathcal{K}(H),\alpha\in\mathbb{C}$;

b) $s_n(xb)\leqslant s_n(x)\|b\|_{\mathcal{B}(H)}, s_n(bx)\leqslant s_n(x)\|b\|_{\mathcal{B}(H)}$ for all $x\in\mathcal{K}(H),b\in\mathcal{B}(H)$.
\end{claim}

Let $\mathcal{F}(H)$ be a two-sided ideal in $\mathcal{B}(H)$ of
all operators with finite range and let $\mathcal{I}$ be an
arbitrary proper two-sided ideal in $\mathcal{B}(H)$. Then
$\mathcal{I}$ is a $*$-ideal (\cite{K-R1}, Ch.6, \S 8, Proposition
6.8.9) and the following inclusion holds:
$\mathcal{F}(H)\subseteq\mathcal{I}$ (\cite{K-R1}, Ch.6, \S 8,
Theorem 6.8.3), in particular, $\mathcal{I}$ contains all
finite-dimensional projections from $\mathcal{P}(H)$. If $H$ is a
separable Hilbert space, then the converse inclusion
$\mathcal{I}\subseteq\mathcal{K}(H)$ also holds (\cite{Calkin},
Theorem 1.4). If, however, $H$ is not separable, then for proper
two-sided ideals in  $\mathcal{B}(H)$ we have the following
proposition.

\begin{claim}\label{non_separable}(\cite{Gar}, Proposition 1)  (i) $\mathcal{D}=\{x\in\mathcal{B}(H):x(H) \mbox{ is separable }\}$ is a proper two-sided ideal in $\mathcal{B}(H)$, in addition $\mathcal{K}(H)\subset\mathcal{D}$;

(ii) If $\mathcal{I}$ is an ideal in $\mathcal{B}(H)$, then either $\mathcal{I}\subseteq\mathcal{K}(H)$ or $\mathcal{D}\subseteq\mathcal{I}$.
\end{claim}

Let $X$ be a linear space over the field $\mathbb{C}$. A function
$\|\cdot\|$ from $X$ to $\mathbb{R}$ is a quasi-norm, if for all
$x,y\in X,\alpha \in\mathbb{C}$ the following properties hold:

1) $\|x\|\geqslant 0, \|x\|=0 \Leftrightarrow x=0$;

2) $\|\alpha x\|=|\alpha|\|x\|$;

3) $\|x+y\|\leqslant C(\|x\|+\|y\|), C\geqslant 1$;

The couple $(X,\|\cdot\|)$ is a quasi-normed space and the least
of all constants $C$  satisfying the  inequality 3) above is
called the modulus of concavity of the quasi-norm $\|\cdot\|$.

It is known (see e.g. \cite{Kalton}, \S 1) that for each
quasi-norm $\|\cdot\|$ on $X$ there exists an equivalent
$p$-additive quasi-norm $\interleave\cdot\interleave$, that is a
quasi-norm  $\interleave\cdot\interleave$ on $X$ satisfying the
following property of $p$-additivity: $\interleave
x+y\interleave^p\leqslant\interleave x\interleave^p+\interleave
y\interleave^p$, where $p$ is such that $C=2^{\frac{1}{p}-1}$, in
particular, $0<p\leqslant 1$ since $C\geqslant 1$. In this case,
the function $d: X^2\to\mathbb{R}$ defined by $d(x,y):=\interleave
x-y\interleave^p$, $x,y\in X$ is an invariant metric on $X$, and
in the topology $\tau_d$, generated by the metric $d$, the linear
space $X$ is a topological vector space. If $(X,d)$ is a complete
metric space, then $(X,\|\cdot\|)$ is called a quasi-Banach space
and the quasi-norm $\|\cdot\|$ is a complete quasi-norm; in this
case, $(X,\tau_d)$ is an $F$-space.

\begin{claim}\label{qb} Let $(X, \|\cdot\|)$ be a quasi-Banach space with the modulus of concavity $C$, let
$\interleave\cdot\interleave$ be a $p$-additive quasi-norm
equivalent to the quasi-norm $\|\cdot\|, C=2^{\frac{1}{p}-1}$. If
$x_n\in X$, $n\ge 1$ and $\sum\limits_{n=1}^\infty\interleave x_n
\interleave^p<\infty$, then  the series $\sum\limits_{n=1}^\infty
x_n$  converges in $(X, \|\cdot\|)$, i.e. there exists  $x\in X$
such that $\bigl\|x-\sum\limits_{n=1}^k x_n\bigl\|\rightarrow 0$
for $k\rightarrow \infty.$
\end{claim}

\begin{proof} For partial sums $S_k=\sum\limits_{n=1}^k x_n$ we have
$$d(S_{k+l},S_k)=\interleave S_{k+l}-S_k\interleave^p=\interleave\sum\limits_{n=l+1}^{k+l} x_n\interleave^p\leqslant \sum\limits_{n=l+1}^{k+l} \interleave x_n\interleave^p\rightarrow 0 \mbox{ for } k,l\rightarrow \infty,$$
 i.e. $\{S_k\}_{k=1}^\infty$ is a Cauchy sequence in $(X,d)$. Since the metric space $(X,d)$ is complete, there exists $x\in X$ such that $d(S_k,x)=\interleave S_k-x\interleave^p\rightarrow 0$ for $k\rightarrow \infty$. Since quasi-norms $\|\cdot\|$ and $\interleave\cdot\interleave$ are equivalent we have that $\|S_k-x\|\rightarrow 0$ for $k\rightarrow \infty$.
\end{proof}

Let $(X,\|\cdot\|_X),(Y,\|\cdot\|_Y)$ be quasi-normed spaces and
$\mathcal{B}(X,Y)$ be the linear space of all bounded linear
mappings $T:X\to Y$. For each $T\in\mathcal{B}(X,Y)$ set
$\|T\|_{\mathcal{B}(X,Y)}=\sup\{\|Tx\|_Y: \|x\|\leqslant 1\}$. As
in the case of normed spaces, the set  $\mathcal{B}(X,Y)$
coincides with the set of all continuous linear mappings from $X$
to $Y$, moreover, the function
$\|\cdot\|_{\mathcal{B}(X,Y)}\colon\mathcal{B}(X,Y)\to\mathbb{R}$
is a quasi-norm on  $\mathcal{B}(X,Y)$ whose modulus of concavity,
 does not exceed the modulus of concavity of the quasi-norm
$\|\cdot\|_Y$. Furthermore,
$\|Tx\|_Y\leqslant\|T\|_{\mathcal{B}(X,Y)}\|x\|_X$ for all
$T\in\mathcal{B}(X,Y)$ and $x\in X$.

\begin{claim} If  $(Y,\|\cdot\|_Y)$ is a quasi-Banach space, then
$(\mathcal{B}(X,Y), \|\cdot\|_{\mathcal{B}(X,Y)})$ is a
quasi-Banach space too. \end{claim}

\begin{proof} Since $\|\cdot\|_Y$ is a quasi-norm on $Y$, there  exists a $p$-additive quasi-norm
$\interleave\cdot\interleave_Y$ equivalent to $\|\cdot\|_Y$, i.e.
$\alpha_1\interleave y \interleave_Y\leqslant
\|y\|_Y\leqslant\beta_1\interleave y \interleave_Y$ for all $y\in
Y$ and some constants $\alpha_1,\beta_1>0$. Similarly, there
exists a $q$-additive quasi-norm
$\interleave\cdot\interleave_{\mathcal{B}(X,Y)}$ equivalent to the
quasi-norm $\|\cdot\|_{\mathcal{B}(X,Y)}$, i.e.
$\alpha_2\interleave T \interleave_{\mathcal{B}(X,Y)}\leqslant
\|T\|_{\mathcal{B}(X,Y)}\leqslant\beta_2\interleave T
\interleave_{\mathcal{B}(X,Y)}$ for all $T\in\mathcal{B}(X,Y)$ and
some $\alpha_2,\beta_2>0, 0<p,q\leqslant 1$.

Let $\{T_n\}_{n=1}^\infty$ be a Cauchy sequence in  $(\mathcal{B}(X,Y),d)$, where
$d(T,S)=\interleave T-S\interleave^q_{\mathcal{B}(X,Y)}, T,S\in\mathcal{B}(X,Y)$.
Fix $\varepsilon>0$ and select a positive integer $n(\varepsilon)$ such that
$\interleave T_n-T_m\interleave^q_{\mathcal{B}(X,Y)}<\varepsilon^q$ for all
$n,m\geqslant n(\varepsilon)$. For every $x\in X$ we have
\begin{multline*}
\interleave T_nx-T_mx\interleave_Y^p\leqslant\frac{1}{\alpha_1^p}
\|T_nx-T_mx\|_Y^p\leqslant\frac{1}{\alpha_1^p} \|T_n-T_m\|_{\mathcal{B}(X,Y)}^p\|x\|_X^p\leqslant
\\\leqslant\left(\frac{\beta_2}{\alpha_1} \right)^p\interleave T_n-T_m\interleave_{\mathcal{B}(X,Y)}^p\|x\|_X^p
<  \left(\frac{\beta_2}{\alpha_1} \right)^p\|x\|_X^p\varepsilon^p \mbox{ for } n,m\geqslant n(\varepsilon).
\end{multline*}
Thus, $\{T_nx\}_{n=1}^\infty$ is a Cauchy sequence in $(Y,d_Y)$,
where $d_Y(x,y)=\interleave x-y\interleave^p_Y$. Since the metric
space $(Y,d_Y)$ is complete, there exists $ T(x)\in Y$ such that
$\interleave T_n(x)-T(x)\interleave_Y^p\rightarrow 0$ for
$n\rightarrow\infty$. The verification that $T\in\mathcal{B}(X,Y)$
and $\interleave T_n-T\interleave^q_{\mathcal{B}(X,Y)}\rightarrow
0$ for $n\rightarrow\infty$ is routine and is therefore omitted.
\end{proof}

Let $\mathcal{I}$ be a nonzero two-sided ideal in $\mathcal{B}(H)$.

A quasi-norm $\|\cdot\|_{\mathcal{I}}:\mathcal{I}\to\mathbb{R}$ is called symmetric quasi-norm if

1) $\|axb\|_{\mathcal{I}}\leqslant \|a\|_{\mathcal{B}(H)}\|x\|_{\mathcal{I}}\|b\|_{\mathcal{B}(H)}$ for all $x\in\mathcal{I}, a,b\in\mathcal{B}(H)$;

2) $\|p\|_{\mathcal{I}}=1$ for any one-dimensional projection $p\in\mathcal{I}$.

\begin{claim}\label{sim qn} Let $\|\cdot\|_{\mathcal{I}}$ be a symmetric quasi-norm on two-sided ideal $\mathcal{I}$. Then

a) $\|x\|_{\mathcal{I}}=\|x^*\|_{\mathcal{I}}=\bigl\| |x|\bigl\|_{\mathcal{I}}$ for all $x\in\mathcal{I}$;

b) If $x\in\mathcal{I}\subset\mathcal{K}(H), y\in\mathcal{K}(H), s_n(y)\leqslant s_n(x), n=1,2,\dots$, then $y\in\mathcal{I}$ and $\|y\|_{\mathcal{I}}\leqslant\|x\|_{\mathcal{I}}$;

c) If $\mathcal{I}\subset\mathcal{K}(H)$, then  $\|x\|_{\mathcal{B}(H)}\leqslant\|x\|_{\mathcal{I}}$ for all $x\in\mathcal{I}$.
\end{claim}

\begin{proof}
a)  Let $x=u|x|$ be the polar decomposition of the operator $x$.
Then $\|x\|_{\mathcal{I}}= \| u|x|\|_{\mathcal{I}}\leqslant
\bigl\| |x|\bigl\|_{\mathcal{I}}$. Since $u^*x=|x|$, the
inequality $\bigl\||x|\bigl\|_{\mathcal{I}}\leqslant
\|x\|_{\mathcal{I}}$ holds and so
$\bigl\||x|\bigl\|_{\mathcal{I}}=\|x\|_{\mathcal{I}}$. Using the
equalities $x^*=|x|u^*, x^*u=|x|$ in the same manner, we obtain
that $\bigl\||x|\bigl\|_{\mathcal{I}}=\|x^*\|_{\mathcal{I}}$.

b) Since $x,y$ are compact operators and $s_n(y)\leqslant s_n(x)$
we have $s_n(y)=\alpha_n s_n(x)$, where
$0\leqslant\alpha_n\leqslant 1, n\in\mathbb{N}$. By
Hilbert-Schmidt theorem, there exists an orthogonal system of
eigenvectors $\{\varphi_n\}_{n=1}^\infty$ for the operator $|y|$
such that $|y|(\varphi)=\sum\limits_{n=1}^\infty
s_n(y)c_n\varphi_n$, where $c_n=(\varphi,\varphi_n),\varphi\in H$.
Since $s_n(y)=\alpha_n s_n(x)$, it follows that
$\mathrm{card}\{\varphi_n\}\leqslant\mathrm{card}\{\psi_n\}$,
where $\{\psi_n\}_{n=1}^\infty$ is an orthogonal system of
eigenvectors for the operator $|x|$. Thus, there exists a unitary
operator $u\in\mathcal{B}(H)$ such that $u(\psi_n)=\varphi_n$, in
addition, $u|x|u^{-1}\geqslant|y|$.

By Proposition \ref{pr}, there exists an operator
$a\in\mathcal{B}(H)$ with $\|a\|_{\mathcal{B}(H)}\leqslant 1$ such
that $|y|=a^*u|x|u^{-1}a$. Consequently, $|y|\in\mathcal{I}$ and
$\bigl\| |y| \bigl\|_{\mathcal{I}}\leqslant \bigl\| |x|
\bigl\|_{\mathcal{I}}$, thus $y\in\mathcal{I}$ and
$\|y\|_{\mathcal{I}}\leqslant\|x\|_{\mathcal{I}}$.

c) Let $y(\cdot)=s_1(x)(\cdot,\varphi)\varphi$, where $\varphi$ is an arbitrary vector in $H$ with $\|\varphi\|_H=1$. Whereas $s_n(y)\leqslant s_n(x)$, we have $\|x\|_{\mathcal{B}(H)}=s_1(x)=\|y\|_{\mathcal{B}(H)}=\|y\|_{\mathcal{I}}\leqslant \|x\|_{\mathcal{I}}$ (see b)).
\end{proof}

A two-sided ideal $\mathcal{I}$ of compact operators from
$\mathcal{B}(H)$ is called a symmetric quasi-Banach (respectively,
Banach) ideal, if  $\mathcal{I}$ is equipped with a complete
symmetric quasi-norm (respectively, norm).

Let $\mathcal{I,J}$ be two-sided ideals of compact operators from
$\mathcal{B}(H)$. A linear mapping
$\delta:\mathcal{I}\to\mathcal{J}$ is called a derivation, if
$\delta(xy)=\delta(x)y+x\delta(y)$ for all $x,y\in\mathcal{I}$.
If, in addition, $\delta(x^*)=(\delta(x))^*$ for all
$x\in\mathcal{I}$, then $\delta$ is called a $*$-derivation.

For each derivation $\delta:\mathcal{I}\to\mathcal{J}$ define
mappings
$\delta_\mathrm{Re}(x):=\frac{\delta(x)+\delta(x^*)^*}{2}$ and
$\delta_\mathrm{Im}(x):=\frac{\delta(x)-\delta(x^*)^*}{2i},
x\in\mathcal{I}$. It is easy to see that $\delta_\mathrm{Re}$ and
$\delta_\mathrm{Im}$ are *-derivations from $\mathcal{I}$ to
$\mathcal{J}$, moreover
$\delta=\delta_\mathrm{Re}+i\delta_\mathrm{Im}$.

If $a\in \mathcal{B}$(H), then the mapping $\delta
_{a}:\mathcal{B}(H)\to\mathcal{B}(H)$ given by $\delta _{a}( x):
=[a,x]=ax-xa$, $x\in \mathcal{B}(H)$, is a derivation. Derivations
of this type are called inner. When $\mathcal{I}$ is a two-sided
ideal in $\mathcal{B}(H)$, then $\delta_a(\mathcal{I})\subset
\mathcal{I}$ for  all $a\in\mathcal{B}(H)$. If $\mathcal{J}$ is
also a two-sided ideal in $\mathcal{B}(H)$ and $a\in\mathcal{J}$,
then $\delta_a(\mathcal{I})\subset \mathcal{I\cap J}.$

\section{The set Der$(\mathcal{I},\mathcal{J})$ for symmetric quasi-Banach ideals $\mathcal{I}$ and
$\mathcal{J}$}
The following theorem gives a positive answer to question
\ref{q3}.
\begin{theorem}\label{th6}
Let $\mathcal{I, J}$ be symmetric quasi-Banach ideals of compact operators from
$\mathcal{B}(H)$ and $\delta
$ is a derivation from $\mathcal{I}$ to $\mathcal{J}$. Then $\delta$ is a continuous mapping from $\mathcal{I}$ to $\mathcal{J}$, i.e. $\delta\in\mathcal{B(I,J)}$.
\end{theorem}

\begin{proof} Without loss of generality we may assume that  $\delta$ is a $*$-derivation.
The spaces
$(\mathcal{I},\|\cdot\|_{\mathcal{I}}),
(\mathcal{J},\|\cdot\|_{\mathcal{J}})$ are  $F$-spaces,  and
therefore it is sufficient to prove that the graph of $\delta$ is
closed. Suppose a contrary, that is there exists a sequence
$\{x_n\}_{n=1}^\infty\subset \mathcal{I}$ such that
$\|\cdot\|_{\mathcal{I}}-\lim\limits_{n\rightarrow\infty}x_n=0$
and
$\|\cdot\|_{\mathcal{J}}-\lim\limits_{n\rightarrow\infty}\delta(x_n)=x\neq
0$.

Since $x_n=\mathrm{Re}x_n+i\mathrm{Im}x_n$ for all $n\in\mathbb{N}$, where
$\mathrm{Re}x_n=\frac{x_n+x_n^*}{2}, \mathrm{Im}x_n=\frac{x_n-x_n^*}{2},$ and
$\|x_n\|_\mathcal{I}\rightarrow 0, \|x_n^*\|_\mathcal{I}=\|x_n\|_\mathcal{I}\rightarrow 0$, we have
$$\|\mathrm{Re}x_n\|_\mathcal{I}=\left\| \frac{x_n+x_n^*}{2}\right\|_\mathcal{I}\leqslant\frac{C(\|x_n\|_\mathcal{I}+\|x_n^*\|_\mathcal{I})}{2}      \rightarrow 0 $$
and
$$\|\mathrm{Im}x_n\|_\mathcal{I}=\left\| \frac{x_n-x_n^*}{2}\right\|_\mathcal{I}\leqslant\frac{C(\|x_n\|_\mathcal{I}+\|x_n^*\|_\mathcal{I})}{2}  \rightarrow 0,$$
 where $C$ is the modulus of concavity of the quasi-norm $\|\cdot\|_{\mathcal{I}}$.
 Consequently, we may assume  that $x_n^*=x_n$ for all $n\in\mathbb{N}$. In this case, from the
 relationships $$x\xleftarrow{\|\cdot\|_\mathcal{J}} \delta(x_n)=
 \delta(x_n^*)=\delta(x_n)^*\xrightarrow{\|\cdot\|_\mathcal{J}} x^*,$$ we obtain $x=x^*$.

Writing $x=x_+-x_-$, where $x_+,x_-\geqslant 0$ and $x_+ x_-=0$,
we may assume that $x_+\neq 0$, otherwise we consider the sequence
$\{-x_n\}_{n=1}^\infty$. Since $x_+$ is a nonzero positive compact
operator, $\lambda=\|x_+\|_{\mathcal{B}(H)}$ is an eigenvalue of
$x_+$ corresponding to a finite-dimensional eigensubspace. Let $q$
be a projection onto this subspace.

Fix an arbitrary non-zero vector $\varphi\in q(H)$ and consider
the projection $p$ onto the one-dimensional subspace spanned by
$\varphi$. Combining the inequality $p\leqslant q$ with the
equality $qx_+q=\lambda q$, we obtain $pxp=pqxqp
=\lambda pqp=\lambda p$. Replacing, if necessary, the sequence
$\{x_n\}_{n=1}^\infty$ with the sequence
$\{\frac{x_n}{\lambda}\}_{n=1}^\infty$, we may assume
\begin{equation}\label{pxp}
 pxp
=p.
\end{equation}

Since $p$ is one-dimensional, it follows that $pap=\alpha
p,\alpha\in\mathbb{C}$ for any operator $a\in\mathcal{B}(H)$, in
particular, $px_np=\alpha_np$, therefore
$|\alpha_n|=\|px_np\|_\mathcal{I}\rightarrow 0$ for
$n\rightarrow\infty$. Writing
$$\|\delta(p)x_np\|_\mathcal{J} \leqslant \|\delta(p)\|_\mathcal{J} \| x_np \|_{\mathcal{B}(H)} \leqslant\|\delta(p)\|_\mathcal{J} \| x_n \|_{\mathcal{B}(H)} \leqslant \|\delta(p)\|_\mathcal{J} \| x_n \|_\mathcal{I},$$
we infer $\|\delta(p)x_np\|_\mathcal{J}\rightarrow 0$ and
$\|px_n\delta(p)\|_\mathcal{J}=\|(\delta(p)x_np)^*\|_\mathcal{J}\rightarrow
0$.

Since $pxp\begin{smallmatrix}(\ref{pxp})\\
=
\\~
\end{smallmatrix}p\in\mathcal{J}$, we have
\begin{gather*}
\begin{split}
\|\delta(p x_n p)-pxp\|_J &=\|\delta(p)x_np+p\delta(x_n)p+px_n\delta(p)-pxp\|_\mathcal{J}\leqslant
\\
&\leqslant C\|\delta(p)x_np+px_n\delta(p)\|_J+C\|p\delta(x_n)p-pxp\|_J\leqslant
\\
&\leqslant C^2\|\delta(p)x_np\|_J+C^2\|px_n\delta(p)\|_J+C\|p\delta(x_n)p-pxp\|_J
\rightarrow 0,
\end{split}
\end{gather*}
i.e. $\delta(px_np)\xrightarrow{\|\cdot\|_\mathcal{J}} pxp$. Hence
$$p\begin{smallmatrix}(\ref{pxp})\\
=
\\~
\end{smallmatrix}pxp=\|\cdot\|_\mathcal{J}-\lim\limits_{n\rightarrow\infty}\delta(px_np)= \|\cdot\|_\mathcal{J}-\lim\limits_{n\rightarrow\infty}\delta(\alpha_np)= \|\cdot\|_\mathcal{J}-\lim\limits_{n\rightarrow\infty}\alpha_n\delta(p)=0,$$ which is a contradiction, since $p\neq 0$.

Consequently, $\delta$ is a continuous mapping from $(\mathcal{I},\|\cdot\|_\mathcal{I})$ to $(\mathcal{J},\|\cdot\|_\mathcal{J})$.
\end{proof}

Note, that in (\cite{B-S_d}, Theorem 7) a version of Theorem
\ref{th6} is obtained for the case of an arbitrary symmetric
Banach ideal $\mathcal{I}=\mathcal{J}$ in a properly infinite von
Neumann algebra $\mathcal{M}$.

The following theorem gives a positive answer to Question \ref{q1}.
\begin{theorem}\label{th7}
If $\mathcal{I,J}$ are symmetric quasi-Banach ideals of compact operators from
$\mathcal{B}(H)$, then for every derivation
$\delta: \mathcal{I}\rightarrow \mathcal{J}$ there exists an operator
 $a\in\mathcal{B}(H)$ such that
$\delta(\cdot)=\delta_a(\cdot)=[a,\cdot]$, in addition, $\|a\|_{\mathcal{B}(H)}\leqslant\|\delta\|_\mathcal{B(I,J)}$ and $ax\in\mathcal{J}$ for all $x\in\mathcal{I}$.
\end{theorem}
\begin{proof} Fix an arbitrary vector $\varphi_0\in H$ with $\|\varphi_0\|_H=1$ and consider projection
$p_0(\cdot):=(\cdot,\varphi_0)\varphi_0$ onto one-dimensional subspace spanned by $\varphi_0$. Obviously,
$p_0\in\mathcal{I\cap J}$.

Let $x\in\mathcal{I},x(\varphi_0)=0$ and $\varphi\in H$. Since
$$xp_0(\varphi)=x(p_0(\varphi))=
x\bigl((\varphi,\varphi_0)\varphi_0\bigl)=(\varphi,\varphi_0)x(\varphi_0)=0,$$
it follows that $xp_0=0$, and so $\delta(xp_0)(\varphi_0)=0.$
Consequently, the linear operator
$a(z(\varphi_0))=\delta(zp_0)(\varphi_0)$ is correctly defined on
the linear subspace $L:=\{z(\varphi_0):z\in\mathcal{I}\}\subset
H$. If $\varphi\in H, z(\cdot)=(\cdot,\varphi_0)\varphi$, then
$z\in\mathcal{I}$ and $z(\varphi_0)=\varphi$, which implies $L=H$.
The definition of the operator $a$ immediately implies that
$ax\in\mathcal{J}$ for all $x\in\mathcal{I}$.

For arbitrary $z\in\mathcal{B}(H), \varphi\in H$, we have
\begin{gather*}
\begin{split}
|zp_0|^2(\varphi)&=(p_0z^*zp_0)(\varphi)=(p_0z^*z)((\varphi,\varphi_0)\varphi_0)=(\varphi,\varphi_0)p_0(z^*z(\varphi_0))
=
\\
&=(z\varphi_0,z\varphi_0)(\varphi,\varphi_0)\varphi_0=(z\varphi_0,z\varphi_0)p_0(\varphi)=\|z(\varphi_0)\|_H^2p_0(\varphi),
\end{split}
\end{gather*}
in particular, $\|zp_0\|_{\mathcal{B}(H)}=\bigl\| |zp_0|
\bigl\|_{\mathcal{B}(H)}=\bigl\|\|z(\varphi_0)\|_Hp_0
\bigl\|_{\mathcal{B}(H)}=\|z(\varphi_0)\|_H$. Applying this
observation together with Theorem \ref{th6} guaranteing
$\|\delta(x)\|_{\mathcal{J}}\leqslant
\|\delta\|_\mathcal{B(I,J)}\|x\|_{\mathcal{I}}$ for all
$x\in\mathcal{I}$, we have
\begin{gather*}
\begin{split}
\|a(x(\varphi_0))\|_H&=\|\delta(xp_0)(\varphi_0)\|_H=\|\delta(xp_0)p_0\|_{\mathcal{B}(H)} \leqslant\|\delta(xp_0)\|_{\mathcal{B}(H)}\|p_0\|_{\mathcal{B}(H)}\leqslant
\\
&\leqslant \|\delta(xp_0)\|_\mathcal{J}\leqslant\|\delta\|_\mathcal{B(I,J)}\|xp_0\|_\mathcal{I} \leqslant
\\
&\leqslant\|\delta\|_\mathcal{B(I,J)}\|p_0\|_\mathcal{I}\|xp_0\|_{\mathcal{B}(H)}=
\|\delta\|_\mathcal{B(I,J)}\|x(\varphi_0)\|_H.
\end{split}
\end{gather*}
This shows that $a$ is a bounded operator on $H$  and
$\|a\|_{\mathcal{B}(H)}\leqslant\|\delta\|_\mathcal{B(I,J)}$.

Finally, for all  $x,z\in\mathcal{I}$ we have
\begin{gather*}
\begin{split}
[a,x](z(\varphi_0))&=ax(z(\varphi_0))-xa(z(\varphi_0))=a(xz(\varphi_0))-xa(z(\varphi_0))=
\\
&=\delta(xzp_0)(\varphi_0)-x\delta(zp_0)(\varphi_0)=\delta(x)zp_0(\varphi_0)=\delta(x)z(\varphi_0)
\end{split}
\end{gather*}
and since $L=H$, it follows
$\delta(\cdot)=[a,\cdot]=\delta_a(\cdot)$.
\end{proof}

Let $\mathcal{I,J}$ be arbitrary two-sided ideals in $\mathcal{B}(H)$. The set
$$D(\mathcal{I, J})=\{a\in\mathcal{B}(H): \quad ax-xa\in \mathcal{J},\quad \forall x\in\mathcal{I}\}$$
 is called the $\mathcal{J}$-essential commutant of $\mathcal{I}$, and the set
$$\mathcal{J:I}=\{a\in\mathcal{B}(H):  \quad ax\in\mathcal{J},\quad \forall x\in\mathcal{I}\}$$ is called the
$\mathcal{J}$-dual space of $\mathcal{I}$. It is clear that
$\mathcal{J:I}$ is a two-sided ideal in $\mathcal{B}(H)$, in
particular, $xa\in\mathcal{J}$ for all
$x\in\mathcal{I},a\in\mathcal{J:I}$ . If
$\mathcal{I}\nsubseteq\mathcal{J}$, then
$\mathbbm{1}\notin\mathcal{J:I}$, i.e.
$\mathcal{J:I}\neq\mathcal{B}(H)$, and so $\mathcal{J:I}$ is a
proper ideal in $\mathcal{B}(H)$. However, in the case when
$\mathcal{I}\subseteq\mathcal{J}$ we have
$\mathcal{J:I}=\mathcal{B}(H)$, in particular,
$C_r:C_p=\mathcal{B}(H)$ for all $0<p\leqslant r$, where
$C_p=\{x\in\mathcal{K}(H):
\|x\|_p=\bigl(tr(|x|^p)\bigl)^{\frac{1}{p}}<\infty\}$ is the
Schatten ideal of compact operators from $\mathcal{B}(H),0<p<
\infty, tr $ is the standart trace on $\mathcal{B_+}(H)$.

\begin{claim}\label{subset_K(H)} If $\mathcal{I,J}$ are proper two-sided ideals of compact operators in
$\mathcal{B}(H)$ and $\mathcal{I}\nsubseteq\mathcal{J}$, then $\mathcal{J:I}\subset\mathcal{K}(H)$.
\end{claim}
\begin{proof} Since $\mathcal{I}\nsubseteq\mathcal{J}$, $\mathcal{J:I}$ is a proper two-sided ideal in $\mathcal{B}(H)$.
If $H$ is a separable Hilbert space, then
$\mathcal{J:I}\subset\mathcal{K}(H)$ (\cite{Calkin}, Theorem 1.4).
Suppose that $H$ is not separable and
$\mathcal{J:I}\nsubseteq\mathcal{K}(H)$. By Proposition
\ref{non_separable}, the proper two-sided ideal
$\mathcal{D}=\{x\in\mathcal{B}(H):x(H) \mbox{ is separable
}\}\subset\mathcal{J:I}$. Since $\mathcal{I}\nsubseteq\mathcal{J}$
there exists a positive compact operator $a\in\mathcal{I\setminus
J}$. Since $a\in \mathcal{D}$, we have that  $L:=\overline{a(H)}$
is separable. Let $p\in \mathcal{P}(H)$ be the orthogonal
projection onto $L$. Since $a\notin\mathcal{J}$, it follows that
$L$ is infinite-dimensional subspace. Indeed, if it were not the
case, then $a$ would be a finite rank operator and automatically
belonging to $a\in\mathcal{J}$. Therefore
$p\in\mathcal{D}\setminus\mathcal{K}(H)\subset\mathcal{J:I}$, in
addition, $0\neq a=pap\in
(p\mathcal{I}p)\setminus(p\mathcal{J}p)$, i.e.
$p\mathcal{I}p\nsubseteq p\mathcal{J}p$. Since $L$ is a separable
Hilbert space, we have
$(p\mathcal{J}p):(p\mathcal{I}p)\subset\mathcal{K}(L)$.

Let $y\in p\mathcal{I}p$, i.e. $y=py'p$ for some $y'\in\mathcal{I}$. Since $p\in\mathcal{D}\subset\mathcal{J:I}$ we have $py'\in\mathcal{J}$, hence, $p(py')p\in p\mathcal{J}p.$ Consequently, $p\in(p\mathcal{J}p):(p\mathcal{I}p)$, i.e. $p$ is a compact operator in $L$, which is a contradiction. Thus, $\mathcal{J:I}\subset\mathcal{K}(H)$.
\end{proof}

For arbitrary two-sided ideals $\mathcal{I,J}$ in $\mathcal{B}(H)$
we denote by $d(\mathcal{I,J})$ the set of all derivations
$\delta$ from $\mathcal{B}(H)$ to $\mathcal{B}(H)$ such that
$\delta(\mathcal{I})\subset\mathcal{J}.$  Obviously,
$d(\mathcal{I,J})\subset Der(\mathcal{I,J})$. To characterize the set $d(\mathcal{I,J})$ we need the following
theorem.

\begin{theorem}(\cite{Hoffman}, Theorem 1.1)\label{th8}
$D(\mathcal{I, J})=\mathcal{J:I}+\mathbb{C}\mathbbm{1}$.
\end{theorem}

It should be noted that  Theorem \ref{th8} holds for arbitrary von
Neumann algebras, i.e. for any two-sided ideals $\mathcal{I,J}$ in
von Neumann algebra $\mathcal{M}$ we have $D(\mathcal{I,
J})=\mathcal{J:I}+Z(\mathcal{M})$, where $Z(\mathcal{M})$ is the
center of  $\mathcal{M}$(\cite{B-S}, Corollary 5).

\begin{claim}\label{pr9}
$d(\mathcal{I,J})=\{\delta_a: a\in D(\mathcal{I,J})\}$.
\end{claim}
\begin{proof}
Let $\delta_a(\cdot)=[a,\cdot]$ be inner derivation on $\mathcal{B}(H)$ generated by the operator $a\in D(\mathcal{I,J})$. For all $x\in\mathcal{I}$ we have $\delta_a(x)=[a,x]=ax-xa\in \mathcal{J}$, i.e. $\delta_a\in d(\mathcal{I,J})$.

Conversely, let $\delta\in d(\mathcal{I,J})$. Since $\delta$ is a derivation from $\mathcal{B}(H)$ to $\mathcal{B}(H)$ there exists an operator  $a\in\mathcal{B}(H)$ such that $\delta=\delta_a$. If $x\in\mathcal{I}$, then  $[a,x]=\delta(x)\in\mathcal{J}$, i.e. $a\in D(\mathcal{I,J})$.
\end{proof}

Now, let $\mathcal{I,J}$ be arbitrary symmetric quasi-Banach ideals of compact operators from $\mathcal{B}(H)$. According to Theorem \ref{th7}, for each derivation $\delta\in Der(\mathcal{I,J})$ there exists an operator $a\in\mathcal{B}(H)$ such that $\delta(x)=\delta_a(x)=[a,x]$ and $ax\in\mathcal{J}$ for all $x\in\mathcal{I}$, i.e. $a\in\mathcal{J:I}$. Since $a\in\mathcal{J:I}$ we have $\delta\in d(\mathcal{I,J})$ (see Theorem \ref{th8}). Thus $Der(\mathcal{I,J})=d(\mathcal{I,J})$ for any symmetric quasi-Banach ideals $\mathcal{I,J}$ of compact operators.
Conversely, if $a\in \mathcal{J:I}$, then $a\in D(\mathcal{I,J})$ (see Theorem \ref{th8}), consequently $\delta_a\in Der(\mathcal{I,J})$ (see Proposition \ref{pr9}).

Hence, the following theorem holds.

\begin{theorem}\label{pr10}
For arbitrary symmetric quasi-Banach ideals $\mathcal{I,J}$ of
compact operators in  $\mathcal{B}(H)$ each derivation
$\delta:\mathcal{I}\to\mathcal{J}$ has a form $\delta=\delta_a$
for some $a\in \mathcal{J:I}$, in addition
$\|a\|_{\mathcal{B}(H)}\leqslant\|\delta_a\|_\mathcal{B(I,J)}$.
Conversely, $\delta_a\in Der(\mathcal{I,J})$ for all
$a\in\mathcal{J:I}.$
\end{theorem}

If $0<r<p<\infty$, then we have $C_r:C_p=C_q$, where
$\frac{1}{q}=\frac{1}{r}-\frac{1}{p}$ (\cite{Hoffman}, Proposition
5.6). Therefore, the following corollary follows immediately from
Theorem \ref{pr10}.
\begin{corollary}
If $0<p\leqslant r<\infty$, then the mapping $\delta:C_p\to C_r$
is a derivation if and only if $\delta=\delta_a$ for some $a\in
\mathcal{B}(H)$. If $0<r<p< \infty$, then the mapping
$\delta:C_p\to C_r$ is a derivation if and only if
$\delta=\delta_a$ for some $a\in C_q$, where
$\frac{1}{q}=\frac{1}{r}-\frac{1}{p}$.

\end{corollary}


\section{The $\mathcal{J}$-dual space of $\mathcal{I}$ for symmetric quasi-Banach ideals $\mathcal{I}$ and $\mathcal{J}$}
In this section we show that any symmetric quasi-Banach ideal
$(\mathcal{I},\|\cdot\|_\mathcal{I})$ of compact operators from
$\mathcal{B}(H)$ has a form of $\mathcal{I}=C_{E_\mathcal{I}}$
with the quasi-norm
$\|\cdot\|_\mathcal{I}=\|\cdot\|_{C_{E_\mathcal{I}}}$ for a
special symmetric quasi-Banach sequence space
$(E_\mathcal{I},\|\cdot\|_{E_\mathcal{I}})$ in $c_0$ constructed
by $\mathcal{I}$  with the help of Calkin
correspondence. The equality
$\mathcal{J:I}=C_{E_\mathcal{J}:E_\mathcal{I}}$ established in
this section provides a full description of all derivations $\delta\in Der(\mathcal{I,J})$ in terms of $E_\mathcal{J}$-dual space
$E_\mathcal{J}:E_\mathcal{I}$ of $E_\mathcal{I}$ of symmetric
quasi-Banach sequence spaces $E_\mathcal{I}$ and $E_\mathcal{J}$
in $c_0$.

A quasi-Banach lattice $E$ is a vector lattice with a complete
quasi-norm $\|\cdot\|_E$, such that  $\|a\|_E\leqslant\|b\|_E$
whenever  $a,b\in E$ and $|a|\leqslant|b|$. In this case, $\bigl\|
|a| \bigl\|_E=\|a\|_E$ for all $a\in E$ and the
 lattice operations $a\vee b$ and $a\wedge b$ are continuous in the topology $\tau_d$,
 generated by the metric $d(a,b)=\interleave a-b \interleave_E^p$, where
 $\interleave\cdot\interleave_E$ is a $p$-additive quasi-norm equivalent to the quasi-norm $\|\cdot\|_E$.
 Consequently, the set $E_+=\{a\in E:a\geqslant 0\}$ is closed in $(E,\tau_d)$.
 Thus, for any increasing sequence $\{a_k\}_{k=1}^\infty\subset E$ converging in the topology $\tau_d$
 to some $a\in E$, we have $a=\sup\limits_{k\geq 1}a_k$ (\cite{Schaefer}, Ch.\setcounter{gl}{5}\Roman{gl},\S 4).

A sequence $\{a_n\}_{n=1}^\infty$ from a vector lattice $ E$ is
said to be $(r)$-convergent to  $a\in E$ (notation:
$a_n\xrightarrow{(r)} a$) with the regulator $b\in E_+$, if and
only if there exists a sequence of positive numbers
$\varepsilon_n\downarrow 0$ such that
$|a_n-a|\leqslant\varepsilon_n b$ for all $n\in\mathbb{N}$ (see
e.g. \cite{V}, Ch.\setcounter{gl}{3}\Roman{gl}, \S 11).

Observe, that in any quasi-Banach lattice $(E,\|\cdot\|_E)$ it
follows from $a_n\xrightarrow{(r)} a, a_n,a\in E$ that
$\|a_n-a\|_E\rightarrow 0$.

The following proposition is a quasi-Banach version of the
well-known criterion of sequential convergence in Banach lattices.

\begin{claim}\label{Vulih} (compare \cite{V}, Ch. \setcounter{gl}{7}\Roman{gl}, Theorem \setcounter{gl}{7}\Roman{gl}.2.1)
Let $(E, \|\cdot\|_E)$ be a quasi-Banach lattice, $a,a_n\in E$. The following conditions are equivalent:

(i) $\|a_n-a\|_E\rightarrow 0$ for $n\rightarrow\infty$;

(ii) for any subsequence $a_{n_k}$ there exists a subsequence
$a_{n_{k_s}}$ such that $a_{n_{k_s}}\xrightarrow{(r)}a$.
\end{claim}
\begin{proof}
Without loss of generality we may assume that $a=0$.

$(i)\Rightarrow(ii)$ For an equivalent $p$-additive quasi-norm
$\interleave\cdot\interleave_E$ we have
$\interleave~|a_n|~\interleave_E\rightarrow 0$ for
$n\rightarrow\infty$. Hence, we may choose an increasing sequence
of positive integers $n_1<n_2<\dots<n_k<\dots$ such that
$\interleave~|a_{n_k}|~\interleave^p\leqslant \frac{1}{k^3}$. The
estimate
$$\sum\limits_{k=1}^\infty \interleave k^{\frac{1}{p}}|a_{n_k}| \interleave^p=
\\\sum\limits_{k=1}^\infty k\interleave~|a_{n_k}|~\interleave^p\leqslant\sum\limits_{k=1}^\infty \frac{1}{k^2}<\infty,$$
shows that  the series $\sum\limits_{k=1}^\infty
k^{\frac{1}{p}}|a_{n_k}|$ converges in  $(E,\|\cdot\|_E)$ to some
$b\in E_+$  (see Proposition \ref{qb}) and therefore there exists
$b=\sup\limits_{n\geq
1}\sum\limits_{k=1}^nk^{\frac{1}{p}}|a_{n_k}|$ such that we also
have
 $k^{\frac{1}{p}}|a_{n_k}|\leqslant b$
for all $k\in\mathbb{N}$. In particular, $ |a_{n_k}|\leqslant
k^{-\frac{1}{p}} b$, which immediately implies
$a_{n_k}\xrightarrow{(r)} 0$. The same reasoning may be repeated
for any subsequence $\{a_{n_k}\}_{k=1}^\infty$.

The proof of the implication $(ii)\Rightarrow(i)$ is the verbatim
repetition of the analogous result for Banach lattices (\cite{V},
Ch. \setcounter{gl}{7}\Roman{gl}, Theorem
\setcounter{gl}{7}\Roman{gl}.2.1).
\end{proof}

Let $L_1(0,\infty)$ be the Banach space of all integrable
functions on $(0,\infty)$ with the norm
$\|f\|_1:=\int\limits_0^\infty|f|dm$ and $L_\infty(0,\infty)$  be
the Banach space of all essentially bounded measurable functions
on $(0,\infty)$ with the norm
$\|f\|_\infty:=\mathrm{esssup}\{|f(t)|: 0<t<\infty\}$). For each
$f\in L_1(0,\infty)+L_\infty(0,\infty)$ we define the decreasing
rearrangement $f^*$ of $f$ by setting
$$f^*(t):=\inf\bigl\{s>0:m(\{|f|>s\})\leqslant t\bigl\},t>0.$$
The function $f^*(t)$ is equimeasurable with $|f|$, in particular,
$f^*\in L_1(0,\infty)+L_\infty(0,\infty)$ and $f^*(t)$ is
non-increasing and right-continuous.

We need the following properties of decreasing rearrangements (see
e.g. \cite{K-P-S}, Ch.\setcounter{gl}{2}\Roman{gl}, \S 2).

\begin{claim}\label{rearrag}
Let $f,g\in L_1(0,\infty)+L_\infty(0,\infty)$. We have

(i) if $|f|\leqslant|g|$, then $f^*\leqslant g^*$;

(ii) $(\alpha f)^*=|\alpha|f^*$ for all $\alpha\in\mathbb{R}$;

(iii) if $f\in L_\infty(0,\infty)$, then $(fg)^*\leqslant \|f\|_\infty g^*$;

(iv) $(f+g)^*(t+s)\leqslant f^*(t)+g^*(s)$;

(v) if $fg\in L_1(0,\infty)+L_\infty(0,\infty)$, then $(fg)^*(t+s)\leqslant f^*(t)g^*(s)$.
\end{claim}

Let $l_\infty$ be the Banach lattice of all bounded real-valued
sequences $\xi:=\{\xi_n\}_{n=1}^\infty$ equipped with the norm
$\|\xi\|_\infty=\sup\limits_{n\geqslant 1} |\xi_n|$. For each
$\xi=\{\xi_n\}_{n=1}^\infty\in l_\infty$ the function
$f_\xi(t):=\sum\limits_{n=1}^\infty \xi_n\chi_{[n-1,n)}(t),t>0$ is
contained in $L_\infty(0,\infty)$. For the decreasing
rearrangement $f_\xi^*$, we obviously have
$f_\xi^*(t)=\sum\limits_{n=1}^\infty \xi_n^*\chi_{[n-1,n)}(t),
t>0$, where $\xi^*:=\{\xi_n^*\}_{n=1}^\infty$ is  the decreasing
rearrangement of the sequence $\{|\xi_n|\}_{n=1}^\infty$.
%
%
By Proposition \ref{rearrag} (i),(ii) we have
$\xi^*\leqslant\eta^*$ for $\xi,\eta\in l_\infty$ with
$|\xi|\leqslant|\eta|$, and $(\alpha\xi)^*=|\alpha|\xi^*$,
$\alpha\in \mathbb{R}$.

A linear subspace $\{0\}\neq E\subset l_\infty$ is said to be
solid rearrangement-invariant, if for every $\eta\in E$ and every
$\xi\in l_\infty$ the assumption $\xi^*\leqslant\eta^*$ implies
that $\xi\in E$. Every solid rearrangement-invariant space $E$
contains the space $c_{00}$ of all finitely supported sequences
from $c_0$. If  $E$ contains an element
$\{\xi_n\}_{n=1}^\infty\notin c_0$, then $E=l_\infty$. Thus, for
any solid rearrangement-invariant space $E\neq l_\infty$ the
embeddings  $c_{00}\subset E\subset c_0$ hold.

A solid rearrangement-invariant space $E$ equipped with a complete
quasi-norm (norm) $\|\cdot\|_E$ is called symmetric quasi-Banach
(Banach) sequence space, if

1) $\|\xi\|_E\leqslant\|\eta\|_E$, provided $\xi^*\leqslant\eta^*,
\xi,\eta\in E$;

2) $\|\{1,0,0,\dots\}\|_E=1$.

The inequality $\|a\xi\|_E\leqslant\|a\|_\infty\|\xi\|_E$ for all
$a\in l_\infty, \xi\in E$  immediately follows from Proposition
\ref{rearrag} (iii). In particular, if $E=l_\infty$, then the norm
$\|\cdot\|_E$ is equivalent to $\|\cdot\|_\infty$; for example,
this is the case for any Lorentz space $(l_\psi,\|\cdot\|_\psi)$,
where $\psi\colon[0,\infty)\to \mathbb{R}$ is an arbitrary
nonnegative increasing concave function with the properties
$\psi(0)=0,\psi(+0)\neq
0,\lim\limits_{t\rightarrow\infty}\psi(t)<\infty$ (see details in
\cite{K-P-S}, Ch.\setcounter{gl}{2}\Roman{gl}, \S 5).

The spaces $(c_0,\|\cdot\|_\infty), (l_p,\|\cdot\|_p), 1\leqslant
p<\infty$ (respectively, $(l_p,\|\cdot\|_p)$ for $0<p<1$), where
$$l_p=\bigl\{\{\xi_n\}_{n=1}^\infty\in c_0: \|\{x_n\}\|_p
=\bigl(\sum\limits_{n=1}^\infty|\xi_n|^p\bigl)^{\frac{1}{p}}<\infty\bigl\}$$
are examples of the classical symmetric Banach (respectively,
quasi-Banach) sequence spaces in $c_0$.

Let $(E,\|\cdot\|_E)$ be a symmetric quasi-Banach sequence space.
For every $\xi=\{\xi_n\}_{n=1}^\infty\in E, m\in\mathbb{N},$ we
set
\begin{gather*}
\sigma_m(\xi)=(\underbrace{\xi_1,\dots,\xi_1}_{\mbox{$m$ times}},\underbrace{\xi_2,\dots,\xi_2}_{\mbox{$m$ times}},\dots),
\\
\eta^{(1)}=(\xi_1,\underbrace{0,\dots,0}_{\mbox{$m-1$ times}},\xi_2,\underbrace{0,\dots,0}_{\mbox{$m-1$ times}},\dots),
\\
\eta^{(2)}=(0,\xi_1,\underbrace{0,\dots,0}_{\mbox{$m-2$ times}},0,\xi_2,\underbrace{0,\dots,0}_{\mbox{$m-2$ times}},\dots),
\\
\dots,
\\
\eta^{(m)}=(\underbrace{0,\dots,0}_{\mbox{$m-1$ times}},\xi_1,\underbrace{0,\dots,0}_{\mbox{$m-1$ times}},\xi_2,\dots).
\end{gather*}

Since $(\eta^{(1)})^*=(\eta^{(2)})^*=\dots=(\eta^{(m)})^*=\xi^*\in
E$, it follows $\eta^{(1)},\dots,\eta^{(m)}\in E$. Consequently,
$\sigma_m(\xi)=\eta^{(1)}+\eta^{(2)}+\dots+\eta^{(m)}\in E,$ i.e.
$\sigma_m$ is a linear operator from $E$ to $E$. In addition, we
have
$\|\sigma_m(\xi)\|_E=\|\eta^{(1)}+\eta^{(2)}+\dots+\eta^{(m)}\|_E\leqslant
C(\|\eta^{(1)}\|_E+\|\eta^{(2)}+\eta^{(3)}+\dots+\eta^{(m)}\|_E)\leqslant
C(\|\eta^{(1)}\|_E+C(\|\eta^{(2)}\|_E+\|\eta^{(3)}+\dots+\eta^{(m)}\|_E))\leqslant
(C+C^2+\dots+C^{m-1})\|\xi\|_E$, where  $C$ is the modulus of
concavity of the quasi-norm $\|\cdot\|_E$, in particular
$\|\sigma_m\|_{\mathcal{B}(E,E)}\leqslant C+C^2+\dots+C^{m-1}$ for
all $m\in\mathbb{N}$.

\begin{claim}\label{sigma}
The inequalities $$(\xi+\eta)^*\leqslant \sigma_2(\xi^*+\eta^*),
(\xi\eta)^*\leqslant \sigma_2(\xi^*\eta^*)$$ hold for all
$\xi=\{\xi_n\}_{n=1}^\infty,\eta=\{\eta_n\}_{n=1}^\infty\in
l_\infty$.
\end{claim}

\begin{proof} Since $f_{\xi+\eta}(t)=\sum\limits_{n=1}^\infty(\xi_n+\eta_n)\chi_{[n-1,n)}(t)=f_\xi(t)+f_\eta(t), t>0$,
 we have by Proposition \ref{rearrag} (iv) that
\begin{gather*}
\sum\limits_{n=1}^\infty(\xi_n+\eta_n)^*\chi_{[n-1,n)}(2t)=f_{\xi+\eta}^*(2t)=(f_\xi+f_\eta)^*(2t)\leqslant\\
\leqslant f_\xi^*(t)+f_\eta^*(t)=\sum\limits_{n=1}^\infty(\xi_n^*+\eta_n^*)\chi_{[n-1,n)}(t)= \sum\limits_{n=1}^\infty(\sigma_2(\xi^*+\eta^*))_n\chi_{[n-1,n)}(2t)
\end{gather*}
for all $t>0$, where
$\{(\sigma_2(\xi^*+\eta^*))_n\}_{n=1}^\infty=\sigma_2(\xi^*+\eta^*)$.
In other words, $(\xi+\eta)^*\leqslant\sigma_2(\xi^*+\eta^*)$. The
proof of the inequality
$(\xi\eta)^*\leqslant\sigma_2(\xi^*\eta^*)$ is very similar (one
needs to use Proposition \ref{rearrag} (v)) and is therefore
omitted.
\end{proof}

For a symmetric quasi-Banach sequence space $(E,\|\cdot\|_E)$, we
set
$$C_E:=\{x\in\mathcal{K}(H): \{s_n(x)\}_{n=1}^\infty\in E\},\quad \|x\|_{C_E}:=\|s_n(x)\|_E, x\in C_E.$$

If $E=l_p$ (respectively, $E=c_0$) then
$C_{l_p}=C_p,\|\cdot\|_{C_{l_p}}=\|\cdot\|_{C_p}, 0<p<\infty$
(respectively,
$C_{c_0}=\mathcal{K}(H),\|\cdot\|_{C_{c_0}}=\|\cdot\|_{\mathcal{B}(H)}$).

A quasi-Banach vector sublattice $(E,\|\cdot\|_E)$ in $l_\infty$
is said to be $p$-convex, $0<p<\infty$, if there is a constant
$M$, so that
\begin{equation}\label{p-convex}
\Bigl\|\Bigl(\sum\limits_{i=1}^n|x_i|^p\Bigl)^{\frac{1}{p}}\Bigl\|_E\leqslant
M\Bigl(\sum\limits_{i=1}^n\|x_i\|_E^p\Bigl)^{\frac{1}{p}}
\end{equation}
for every finite collection $\{x_i\}_{i=1}^n\subset E, n\in\mathbb{N}$.

If the estimate (\ref{p-convex}) holds for elements from a
symmetric quasi-Banach ideal $(\mathcal{I},\|\cdot\|_\mathcal{I})$
of compact operators from $\mathcal{B}(H)$, then the ideal
$(\mathcal{I},\|\cdot\|_\mathcal{I})$ is said to be $p$-convex. As
already stated in Theorem \ref{K_S}, for every symmetric Banach
(respectively, symmetric $p$-convex quasi-Banach, $0<p<\infty$)
sequence space $E$ in $c_0$ the couple $(C_E,\|\cdot\|_{C_E})$ is
a symmetric Banach  (respectively, $p$-convex symmetric
quasi-Banach) ideal of compact operators in $\mathcal{B}(H)$.

Thus, for every symmetric Banach ($p$-convex quasi-Banach)
sequence space $(E,\|\cdot\|_E)$ the corresponding symmetric
Banach ($p$-convex quasi-Banach) ideal $(C_E,\|\cdot\|_{C_E})$ of
compact operators from $\mathcal{B}(H)$ is naturally constructed.
This extends the classical Calkin correspondence \cite{Calkin}.

Conversely, if $(\mathcal{I},\|\cdot\|_\mathcal{I})$ is a
symmetric quasi-Banach ideal $(\mathcal{I},\|\cdot\|_\mathcal{I})$
of compact operators from $\mathcal{B}(H)$, then it is of the form
$C_E$ with $\|\cdot\|_\mathcal{I}=\|\cdot\|_{C_E}$ for the
corresponding symmetric quasi-Banach sequence space
$(E_\mathcal{I},\|\xi\|_{E_\mathcal{I}})$. The definition of the
latter space is given below.

Denote by $E_\mathcal{I}$ the set of all $\xi\in c_0$, for which there exists some
$x\in\mathcal{I}$,  such that $\xi^*=\{s_n(x)\}_{n=1}^\infty$. For $\xi\in E_{\mathcal{I}}$ with
$\xi^*=\{s_n(x)\}_{n=1}^\infty,x\in\mathcal{I}$ set $\|\xi\|_{E_\mathcal{I}}=\|x\|_\mathcal{I}$.

Fix an orthonormal set $\{e_n\}_{n=1}^\infty$ in $H$ and for every
 $\xi=\{\xi_n\}_{n=1}^\infty\in c_0$ consider the diagonal operator $x_\xi\in\mathcal{K}(H)$ defined as follows
$$x_\xi(\varphi)=\sum\limits_{n=1}^\infty\xi_nc_n(\varphi)e_n,$$ where $c_n(\varphi)=(\varphi,e_n), \varphi\in H$.
If $\xi\in E_\mathcal{I}$, then $\xi^*=\{s_n(x)\}_{n=1}^\infty$
for some $x\in\mathcal{I}$, and due to equalities
$\{s_n(x_{\xi^*})\}_{n=1}^\infty=\{\xi_n^*\}_{n=1}^\infty=\{s_n(x)\}_{n=1}^\infty$
we have $x_{\xi^*}\in\mathcal{I}$ and
$\|x_{\xi^*}\|_\mathcal{I}=\|x\|_\mathcal{I}=\|\xi\|_{E_\mathcal{I}}$
( see Proposition \ref{sim qn} b)). Moreover, since
$\{s_n(x_\xi)\}_{n=1}^\infty=\{s_n(x_{\xi^*})\}_{n=1}^\infty$ and
$x_{\xi^*}\in\mathcal{I}$, it follows that $x_\xi\in\mathcal{I}$
and $\|\xi\|_{E_\mathcal{I}}=\|x_\xi\|_\mathcal{I}$. Thus, a
sequence $\xi\in c_0$ is contained in $E_\mathcal{I}$, if and only
if operators $x_\xi$ and $x_{\xi^*}$ are in $\mathcal{I}$, in
addition,
$\|\xi\|_{E_\mathcal{I}}=\|x_{\xi^*}\|_\mathcal{I}=\|x_\xi\|_\mathcal{I}$.
In particular, if $\eta\in c_0,\xi\in E_\mathcal{I},
\eta^*\leqslant\xi^*$, then $\eta\in E_\mathcal{I}$ and
$\|\eta\|_{E_\mathcal{I}}\leqslant\|\xi\|_{E_\mathcal{I}}$.

\begin{theorem}\label{inverse K-S} For any symmetric quasi-Banach ideal $\mathcal{I}$ of compact operators from $\mathcal{B}(H)$ the couple $(E_\mathcal{I},\|\cdot\|_{E_\mathcal{I}})$ is a symmetric quasi-Banach sequence space in $c_0$ with the modulus of concavity which does not exceed the modulus of concavity of the quasi-norm $\|\cdot\|_\mathcal{I}$,
in addition, $C_{E_\mathcal{I}}=\mathcal{I}$ and
$\|\cdot\|_{C_{E_\mathcal{I}}}=\|\cdot\|_\mathcal{I}$.
\end{theorem}
\begin{proof} If $\xi,\eta\in E_\mathcal{I}$, then $x_\xi,x_\eta\in\mathcal{I}$, hence $x_\xi+x_\eta\in\mathcal{I}$. Since
$$(x_\xi+x_\eta)(\varphi)=\sum\limits_{n=1}^\infty \xi_nc_n(\varphi)e_n+\sum\limits_{n=1}^\infty \eta_nc_n(\varphi)e_n=\sum\limits_{n=1}^\infty (\xi_n+\eta_n)c_n(\varphi)e_n=x_{\xi+\eta}(\varphi), \varphi\in H,$$
we have $x_{\xi+\eta}\in\mathcal{I}$. Consequently, $\xi+\eta\in E_\mathcal{I}$, moreover, $$\|\xi+\eta\|_{E_\mathcal{I}}=\|x_{\xi+\eta}\|_\mathcal{I}=\|x_{\xi}+x_{\eta}\|_\mathcal{I}\leqslant C(\|x_\xi\|_\mathcal{I}+\|x_\eta\|_\mathcal{I})=C(\|\xi\|_{E_\mathcal{I}}+\|\eta\|_{E_\mathcal{I}}),$$
where $C$ is the modulus of concavity of the quasi-norm $\|\cdot\|_\mathcal{I}$.

Now, let $\xi\in E_\mathcal{I}, \alpha\in\mathbb{R}$. Since $$x_{\alpha\xi}(\varphi)=\sum\limits_{n=1}^\infty \alpha\xi_nc_n(\varphi)e_n=\alpha x_\xi(\varphi), \varphi\in H,$$
we have $\alpha\xi\in E_\mathcal{I}$ and $\|\alpha\xi\|_{E_\mathcal{I}}=\|x_{\alpha\xi}\|_\mathcal{I}=\|\alpha x_\xi\|_\mathcal{I} =|\alpha|\|x_\xi\|_\mathcal{I}=|\alpha|\|\xi\|_{E_\mathcal{I}}$.

It is easy to see that $\|\xi\|_{E_\mathcal{I}}\geqslant 0$ and
$\|\xi\|_{E_\mathcal{I}}= 0\Leftrightarrow\xi=0$.

Hence, $E_\mathcal{I}$ is a solid rearrangement-invariant subspace
in $c_0$ and $\|\cdot\|_{E_\mathcal{I}}$ is a quasi-norm on
$E_\mathcal{I}$.

Let us show that $(E_\mathcal{I},\|\cdot\|_{E_\mathcal{I}})$ is a
quasi-Banach space. Let $\interleave\cdot\interleave_\mathcal{I}$
(respectively, $\interleave\cdot\interleave_{E_\mathcal{I}}$) be a
$p$-additive (respectively, $q$-additive) quasi-norm equivalent to
the quasi-norm $\|\cdot\|_\mathcal{I}$ (respectively,
$\|\cdot\|_{E_\mathcal{I}}), 0<p,q\leqslant 1$.

Let $\xi^{(k)}=\{\xi_n^{(k)}\}_{n=1}^\infty\in E_\mathcal{I}$ and $\interleave\xi^{(k)}-\xi^{(m)}\interleave^p_{E_\mathcal{I}}\rightarrow 0$ for $k,m\rightarrow\infty$. Then $\interleave x_{\xi^{(k)}}-x_{\xi^{(m)}}\interleave^q_\mathcal{I}\rightarrow 0$ for $k,m\rightarrow\infty$, i.e. $x_{\xi^{(k)}}$ is a Cauchy sequence in $(\mathcal{I},d_\mathcal{I})$, where $d_\mathcal{I}(x,y)=\interleave x-y\interleave_\mathcal{I}^q$. Since $(\mathcal{I},d_\mathcal{I})$ is a complete metric space, there exists an operator $x\in\mathcal{I}$ such that $\interleave x_{\xi^{(k)}}-x\interleave^q_\mathcal{I}\rightarrow 0$ for $k\rightarrow\infty$. If $p_n$ is the one-dimensional projection onto subspace spanned by $e_n$, then
\begin{gather*}
\xi^{(k)}p_n=p_n x_{\xi_n^{(k)}}p_n\xrightarrow{\|\cdot\|_\mathcal{I}} p_nxp_n:=\lambda_np_n,\\
0=p_nx_{\xi_n^{(k)}}p_m\rightarrow p_nxp_m, n\neq m.
\end{gather*}
Hence, $x$ is also a diagonal operator, i.e. $x=x_\xi$, where $\xi=\{\lambda_n\}_{n=1}^\infty$. Since  $x\in\mathcal{I}$ we have $\xi\in E_\mathcal{I}$, moreover, $\|\xi^{(k)}-\xi\|_{E_\mathcal{I}}^p=\|x_{\xi^{(k)}}-x_\xi\|_\mathcal{I}\rightarrow 0$  for $k\rightarrow\infty$.

Consequently, $(E_\mathcal{I},\|\cdot\|_{E_\mathcal{I}})$ is a symmetric quasi-Banach sequence space in $c_0$.

Now, let us show that $C_{E_\mathcal{I}}=\mathcal{I}$ and
$\|x\|_{C_{E_\mathcal{I}}}=\|x\|_\mathcal{I}$ for all
$x\in\mathcal{I}$. Let $x\in C_{E_\mathcal{I}}$, i.e.
$\{s_n(x)\}_{n=1}^\infty\in E_\mathcal{I}$. Hence, there exists an
operator $y\in\mathcal{I}$, such that
$s_n(x)=s_n(y),n\in\mathbb{N}$. Consequently, $x\in\mathcal{I}$,
moreover,
$\|x\|_\mathcal{I}=\|\{s_n(x)\}_{n=1}^\infty\|_{E_\mathcal{I}}=\|x\|_{C_{E_\mathcal{I}}}$.
Conversely, if $x\in\mathcal{I}$, then $\{s_n(x)\}_{n=1}^\infty\in
E_\mathcal{I}$ and therefore $x\in C_{E_\mathcal{I}}$.
\end{proof}

The definition of symmetric Banach ($p$-convex quasi-Banach) ideal
$(C_E,\|\cdot\|_{C_E})$ of compact operators from $\mathcal{B}(H)$
jointly with Theorem \ref{inverse K-S} implies the following
corollary:

\begin{corollary}\label{cor_K-S} Let $(E,\|\cdot\|_E)$ be a symmetric Banach
($p$-convex quasi-Banach) sequence space from $c_0$. Then $E_{C_E}=E$ and $\|\cdot\|_{E_{C_E}}=\|\cdot\|_E$
\end{corollary}
\begin{proof} If $\xi\in E$, then $x_{\xi^*}\in C_E$, and due to the equality
$\{s_n(x_{\xi^*})\}_{n=1}^\infty=\xi^*$, we have $\xi\in E_{C_E}$
and $\|\xi\|_{E_{C_E}}=\|x_{\xi^*}\|_{C_E}=\|\xi^*\|_E=\|\xi\|_E$.
The converse inclusion $E_{C_E}\subset E$ may be proven similarly.
\end{proof}

Let $G,F$ be solid rearrangement-invariant  spaces in $c_0$. It is
easy to see that $G$ and $F$ are ideals in the algebra $l_\infty$,
in particular, it follows from the assumptions
$|\xi|\leqslant|\eta|,\xi\in l_\infty,\eta\in G$  that $\xi\in G$,
i.e. $G$ and $F$ are solid linear subspaces in $l_\infty$. We
define $F$-dual space $F:G$ of $G$ by setting
$$F:G=\{\xi\in l_\infty:(\forall \eta\in G) \quad \xi\eta\in F\ \}.$$
It is clear that $F:G$ is an ideal in $l_\infty$ containing
$c_{00}$. If $G\subset F$, then $F:G=l_\infty$, in particular,
$l_\infty:G=l_\infty$ for any solid rearrangement-invariant space
$G$. However, if $G\nsubseteq F$, then $F:G\neq l_\infty$.

\begin{claim}\label{pr_subsetc0} If $F:G\neq l_\infty$, then $F:G\subset c_0$.
\end{claim}
\begin{proof} Suppose that there exists  $\xi=\{\xi_n\}_{n=1}^\infty\in(F:G),\xi\notin c_0$.
Let
$\alpha_n=\mathrm{sign}\xi_n,n\in\mathbb{N},\eta=\{\eta_n\}_{n=1}^\infty\in
G$. Obviously, $\{\alpha_n\eta_n\}_{n=1}^\infty\in G$ and hence,
$|\xi|\eta=\{\xi_n\alpha_n\eta_n\}_{n=1}^\infty\in F$ for all
$\eta\in G$, that is $|\xi|\in (F:G)$,and, in addition,
$|\xi|\notin c_0$. This implies that there exists a subsequence
$0\neq|\xi_{n_k}|\rightarrow\alpha>0$ for $k\rightarrow\infty$.
Consider a sequence $\eta=\{\eta_k\}_{k=1}^\infty$ from
$l_\infty\setminus c_0$ such that $\eta_k=|\xi_{n_k}|$ and show
that $\eta\in F:G$.

For every $\zeta=\{\zeta_n\}_{n=1}^\infty\in G$ define the sequence
$a_\zeta=\{a_n\}_{n=1}^\infty$ such that $a_{n_k}=\zeta_k$ and $a_n=0$,
if $n\neq n_k,k\in\mathbb{N}$. Since $a_\zeta^*=\zeta^*$,
we have $a_\zeta\in G$, and therefore
$\eta\zeta=\{|\xi_{n_k}|\zeta_k\}_{k=1}^\infty=\{|\xi_n|a_n\}_{n=1}^\infty=|\xi|a_\zeta\in F$
for all $\zeta\in G$. Consequently, $\eta=\{\eta_n\}_{n=1}^\infty\in F:G$,
moreover, $\eta_n\geqslant \alpha$ for some $\alpha>0$ and all $n\in\mathbb{N}$.
Since $F:G$ is an ideal in $l_\infty$, it follows that $F:G$ is a solid linear subspace in
$l_\infty$, containing the sequence $\{\eta_n\}_{n=1}^\infty$ with
$\eta_n\geqslant\alpha>0,n\in\mathbb{N}$, that implies $F:G=l_\infty$.
\end{proof}

\begin{claim}\label{lem}If $F:G\neq l_\infty$, then
$F:G=\{\xi\in c_0:  \quad \xi^*\eta^*\in F, \quad \forall \eta\in G\}.$
\end{claim}
\begin{proof} By Proposition \ref{pr_subsetc0}, we have that
$F:G\subset c_0$. Let $\xi=\{\xi_n\}_{n=1}^\infty\in c_0$ and
$\xi^*\eta^*\in F$ for all $\eta\in G$. Due to Proposition
\ref{sigma}, we have $(\xi\eta)^*\leqslant\sigma_2(\xi^*\eta^*)\in
F$, i.e. $(\xi\eta)^*\in F$. Since $F$ is a symmetric sequence
space, it follows that $\xi\eta\in F$ for all $\eta\in G$, i.e.
$\xi\in F:G$.

Conversely, suppose that $\xi=\{\xi_n\}_{n=1}^\infty\in F:G$. Let
$\alpha_n=\mathrm{sign}\xi_n,\eta=\{\eta_n\}_{n=1}^\infty\in G$.
Then $\{\alpha_n\eta_n\}_{n=1}^\infty\in G$, and therefore
$|\xi|\eta=\{\xi_n\alpha_n\eta_n\}_{n=1}^\infty\in F$ for all
$\eta\in G$, i.e. $|\xi|\in F:G\subset c_0$. Since
$|\xi|=\{|\xi_n|\}_{n=1}^\infty\in c_0$, there exists a bijection
of the set $\mathbb{N}$ of natural numbers, such that
$\xi^*=|\xi_{\pi(n)}|$. For linear bijective mapping $U_\pi\colon
l_\infty\to l_\infty$ defined by
$U_\pi(\{\eta_n\}_{n=1}^\infty)=\{\eta_{\pi(n)}\}_{n=1}^\infty$ we
have $U_\pi(\eta\zeta)=U_\pi(\eta)U_\pi(\zeta),
\bigl(U_\pi(\zeta)\bigl)^*=\zeta^*,
\bigl(U^{-1}_\pi(\zeta)\bigl)^*=\zeta^*$ for all $\zeta\in
l_\infty$, in particular, $U_\pi(E)=E$ for any solid
rearrangement-invariant space $E\subset l_\infty$. Consequently,
for all $\eta\in G$ we have
$\xi^*\eta^*=U_\pi(|\xi|)U_\pi(U_\pi^{-1}(\eta^*))=U_\pi(|\xi|U_\pi^{-1}(\eta^*))\in
F$.
\end{proof}

Propositions \ref{pr_subsetc0} and \ref{lem} imply the following
corollary.

\begin{corollary}
$F:G$ is a solid rearrangement-invariant space, moreover, if
$F:G\neq l_\infty$, then $c_{00}\subset F:G\subset c_0$.
\end{corollary}
\begin{proof} The definition of $F:G$ immediately implies that $F:G$ is an ideal in
$l_\infty$ and $c_{00}\subset F:G$. If $F:G\neq l_\infty$, then, due to Proposition \ref{pr_subsetc0},
we have $F:G\subset c_0$.

In the case when $F:G\neq l_\infty$, we have for any $\xi\in
c_0,\eta\in F:G,\xi^*\leqslant\eta^*,\zeta\in G$ that
$\xi^*\zeta^*\leqslant\eta^*\zeta^*\in F$ (see Proposition
\ref{lem}). Consequently, $\xi^*\zeta^*\in F$ for any $\zeta\in
G$, which implies the inclusion $\xi\in F:G$.
\end{proof}

We need some complementary properties of singular values of compact operators.
For every operator $x\in\mathcal{B}(H)$ define the decreasing rearrangement $\mu(x,t)$ of $x$ by setting
$$\mu(x,t)=\inf\{s>0:\mathrm{tr}(|x|>s)\leqslant t\}, t>0$$ (see e.g. \cite{F-K}). If $x\in\mathcal{K}(H)$, then
$$\mu(x,t)=\sum\limits_{n=1}^\infty s_n(x)\chi_{[n-1,n)}(t)=f^*_{\{s_n(x)\}_{n=1}^\infty}(t).$$
In (\cite{F-K}, Lemma 2.5 (v),(vii)) it is established that for every $x,y\in\mathcal{B}(H)$ the inequalities
\begin{gather*}
\mu(x+y,t+s)\leqslant\mu(x,t)+\mu(y,s),\\
\mu(ax,t+s)\leqslant\mu(a,t)\mu(x,s)
\end{gather*}
hold, in particular, if $x,y\in\mathcal{K}(H)$, then
\begin{gather}\label{sv}
\{s_n(x+y)\}_{n=1}^\infty\leqslant\sigma_2\bigl(\{s_n(x)+s_n(y)\}_{n=1}^\infty\bigl),
\\\label{sv2}
\{s_n(xy)\}_{n=1}^\infty\leqslant\sigma_2\bigl(\{s_n(x)s_n(y)\}_{n=1}^\infty\bigl)
\end{gather}

Let $\mathcal{I,J}$ be symmetric quasi-Banach ideals of compact
operators from $\mathcal{B}(H)$ and
$\mathcal{I}\nsubseteq\mathcal{J}$. In this case,
$\mathcal{J:I}\subset\mathcal{K}(H)$ (see Proposition
\ref{subset_K(H)}) and $E_\mathcal{I}\nsubseteq E_\mathcal{J}$
(see Theorem \ref{inverse K-S}), therefore
$E_\mathcal{J}:E_\mathcal{I}\subset c_0$ (see Proposition
\ref{pr_subsetc0}). The following proposition establishes that the
set of operators belonging to the $\mathcal{J}$-dual space
$\mathcal{J:I}$ of $\mathcal{I}$ coincides with the set
$$C_{E_\mathcal{J}:E_\mathcal{I}}=\{x\in\mathcal{K}(H):\{s_n(x)\}_{n=1}^\infty\in
E_\mathcal{J}:E_\mathcal{I}\}.$$

\begin{claim} \label{prmultip}$\mathcal{J:I}=C_{E_\mathcal{J}:E_\mathcal{I}}$.
\end{claim}
\begin{proof}
Let $a\in\mathcal{J:I}$. We claim that $a\in
C_{E_\mathcal{J}:E_\mathcal{I}}$, i.e.
$\xi=\{s_n(a)\}_{n=1}^\infty\in E_\mathcal{J}:E_\mathcal{I}$. For
any sequence $\eta\in E_\mathcal{I}$ consider operators $x_\xi$
and $x_{\eta^*}$. Since $x_\xi\in\mathcal{J:I},
x_{\eta^*}\in\mathcal{I}$, we have $x_\xi
x_{\eta^*}\in\mathcal{J}$. On the other hand, $x_\xi
x_{\eta^*}(\varphi)=\|\cdot\|_H-\lim\limits_{n\rightarrow\infty}\bigl(\sum\limits_{k=1}^n
s_k(a)c_k(x_{\eta^*}(\varphi))e_k\bigl)=\sum\limits_{n=1}^\infty
s_n(a)\eta_n^*c_n(\varphi)e_n=x_{\xi\eta^*}(\varphi)$ for all
$\varphi\in H$. Thus $x_{\xi\eta^*}\in\mathcal{J}$, i.e.
$\xi\eta^*\in E_\mathcal{J}$. Consequently,
$\{s_n(a)\}_{n=1}^\infty\in E_\mathcal{J}:E_\mathcal{I}$ (see
Proposition \ref{lem}) yielding our claim.

Conversely, let $a\in C_{E_\mathcal{J}:E_\mathcal{I}}$, i.e.
$\{s_n(a)\}_{n=1}^\infty\in E_\mathcal{J}:E_\mathcal{I}$.
Due to (\ref{sv2}), for all $x\in\mathcal{I}$ we have
$\{s_n(ax)\}_{n=1}^\infty\leqslant \sigma_2(\{s_n(a)s_n(x)\}_{n=1}^\infty)$.
Since $\{s_n(a)s_n(x)\}_{n=1}^\infty\in E_\mathcal{J}$, it follows that
$\sigma_2(\{s_n(a)s_n(x)\}_{n=1}^\infty)\in E_\mathcal{J}$, and therefore
$\{s_n(ax)\}_{n=1}^\infty\in E_\mathcal{J}$, i.e. $ax\in\mathcal{J}$. Consequently, $a\in\mathcal{J:I}$.
\end{proof}

Let $\mathcal{I,J}$ be symmetric quasi-Banach ideals of compact operators from
 $\mathcal{B}(H),\mathcal{I}\nsubseteq\mathcal{J}$ and $\mathcal{J:I}$ be the
 $\mathcal{J}$-space of $\mathcal{I}$. For any $a\in\mathcal{J:I}$ define a linear mapping
 $T_a\colon\mathcal{I}\to\mathcal{J}$ by setting  $T_a(x)=ax,x\in\mathcal{I}$.

\begin{claim}\label{Ta}
$T_a$ is a continuous linear mapping from $\mathcal{I}$ to
$\mathcal{J}$ for every $a\in\mathcal{J:I}$.
\end{claim}
\begin{proof}
Let $a\in\mathcal{J:I}, \xi=\{s_n(a)\}_{n=1}^\infty,
x_k\in\mathcal{I}$ and $\|x_k\|_\mathcal{I}\rightarrow 0$ for
$k\rightarrow\infty$. Then
$\xi^{(k)}=\{s_n(x_k)\}_{n=1}^\infty\in\ E_\mathcal{I}$ and
$\|\xi^{(k)}\|_{E_\mathcal{I}}\rightarrow 0$. By Proposition
\ref{Vulih}, for every subsequence $\{\xi^{(k_l)}\}_{l=1}^\infty$
there exists a subsequence $\{\xi^{(k_{l_s})}\}_{s=1}^\infty$ such
that $\xi^{(k_{l_s})}\xrightarrow{(r)} 0$ for
$s\rightarrow\infty$, i.e. there exist $0\leqslant\eta\in
E_\mathcal{I}$ and a sequence $\{\varepsilon_s\}_{s=1}^\infty$ of
positive numbers decreasing to zero such that
$|\xi^{(k_{l_s})}|\leqslant\varepsilon_s\eta$. Since
$a\in\mathcal{J:I}$, we have $\zeta=\xi\eta\in E_\mathcal{J}$, in
addition, $\zeta\geqslant 0$. Since
$|\xi\xi^{(k_{l_s})}|\leqslant\varepsilon_s\zeta$, it follows that
$\xi \xi^{(k_{l_s})}\xrightarrow{(r)}0$. By Proposition
\ref{Vulih}, we have $\|\xi\xi^{(k)}\|_{E_\mathcal{J}}\rightarrow
0$. Consequently,
$$\|ax_k\|_\mathcal{J}=\|\{s_n(ax_k)\}\|_{E_\mathcal{J}}\leqslant\|\sigma_2(\xi\xi^{(k)})\|_{E_\mathcal{J}}\leqslant
2C\|\xi\xi^{(k)}\|_{E_\mathcal{J}}\rightarrow 0 \mbox{ for }
k\rightarrow\infty.$$
\end{proof}

By Proposition \ref{Ta}, $T_a$ is a bounded linear operator from
$\mathcal{I}$ to $\mathcal{J}$, therefore
$\|T_a\|_\mathcal{B(I,J)}=\sup\{\|T_a(x)\|_\mathcal{J}:\|x\|_\mathcal{I}\leqslant
1\} =\sup\{\|ax\|_\mathcal{J}:\|x\|_\mathcal{I}\leqslant
1\}<\infty,$ i.e. for all $a\in\mathcal{J:I}$ the quantity
$$\|a\|_\mathcal{J:I}:= \sup\{\|ax\|_\mathcal{J}:
x\in\mathcal{I},\|x\|_\mathcal{I}\leqslant 1\}$$ is well-defined.

\begin{theorem}\label{th_multip}Let $\mathcal{I,J}$ be symmetric quasi-Banach ideals of compact operators
in $\mathcal{B}(H)$ such that $\mathcal{I}\nsubseteq\mathcal{J}$.
Then $(\mathcal{J:I},\|\cdot\|_\mathcal{J:I})$ is a symmetric
quasi-Banach ideal of compact operators whose modulus of concavity
 does not exceed the modulus of concavity of the quasi-norm
$\|\cdot\|_\mathcal{J}$, in addition,
$\|ax\|_\mathcal{J}\leqslant\|a\|_\mathcal{J:I}\|x\|_\mathcal{I}$
for all $a\in\mathcal{J:I},x\in\mathcal{I}$.
\end{theorem}
\begin{proof}
Since $\|\cdot\|_{\mathcal{B(I,J)}}$ is a quasi-norm with the
modulus of concavity which does not exceed the modulus of
concavity of the quasi-norm $\|\cdot\|_\mathcal{J}$, we see that
$\|\cdot\|_\mathcal{J:I}$ is a quasi-norm on  $\mathcal{J:I}$ with
the modulus of concavity which does not exceed the modulus of
concavity of the quasi-norm $\|\cdot\|_\mathcal{J}$.

If $y\in\mathcal{B}(H), a\in\mathcal{J:I}$, then
\begin{gather*}
\begin{split}
\|ya\|_\mathcal{J:I}&=\sup\{\|(ya)x\|_\mathcal{J}: x\in\mathcal{I},\|x\|_\mathcal{I}\leqslant 1\}\leqslant
\\
&\leqslant\sup\{\|y\|_{\mathcal{B}(H)}\|ax\|_\mathcal{J}:x\in\mathcal{I},\|x\|_\mathcal{I}\leqslant 1\}=\|y\|_{\mathcal{B}(H)}\|a\|_\mathcal{J:I}.
\end{split}
\end{gather*}
Since $yx\in\mathcal{I}$ for all $x\in\mathcal{I}$ and $\|yx\|_\mathcal{I}\leqslant\|y\|_{\mathcal{B}(H)}\|x\|_\mathcal{I}$ the for $y\neq 0$ and $\|x\|_\mathcal{I}\leqslant 1$ we have $\|\frac{yx}{\|y\|_{\mathcal{B}(H)}}\|_\mathcal{I}\leqslant 1$. Hence,
\begin{gather*}
\begin{split}
\|ay\|_{\mathcal{J:I}}&=\sup\{\|a(yx)\|_\mathcal{J}:x\in\mathcal{I},\|x\|_\mathcal{I}\leqslant 1\}\leqslant
\\
&\leqslant \|y\|_{\mathcal{B}(H)}\sup\{\|ax\|_\mathcal{J}:x\in\mathcal{I},\|x\|_\mathcal{I}\leqslant 1\}=\|y\|_{\mathcal{B}(H)}\|a\|_\mathcal{I:J}.
\end{split}
\end{gather*}

If $p$ is a one-dimensional projection from $\mathcal{B}(H)$, then $p\in\mathcal{I},\|p\|_\mathcal{I}=1$, and so $$\|p\|_\mathcal{J:I}=\sup\{\|px\|_\mathcal{J}:x\in\mathcal{I},\|x\|_\mathcal{I}\leqslant 1\}\geqslant\|p\|_\mathcal{J}=1.$$
On the other hand, for $x\in\mathcal{I}$ with $\|x\|_\mathcal{I}\leqslant 1$ we have $\|x\|_{\mathcal{B}(H)}\leqslant 1$ (see Proposition \ref{sim qn} c)), and therefore
$$\|px\|_\mathcal{J}=\|p(px)\|_\mathcal{J}\leqslant \|px\|_{\mathcal{B}(H)}\|p\|_\mathcal{J}\leqslant 1.$$
Consequently, $\|p\|_\mathcal{J:I}=1$.

Thus, $\|\cdot\|_\mathcal{J:I}$ is a symmetric quasi-norm on the two-sided ideal $\mathcal{J:I}$. The inequality $\|ax\|_\mathcal{J}\leqslant\|a\|_\mathcal{J:I}\|x\|_\mathcal{I}$ immediately follows from the definition of $\|\cdot\|_\mathcal{J:I}$.

Let us show that $(\mathcal{J:I},\|\cdot\|_\mathcal{J:I})$ is a
quasi-Banach space.

Denote by $\interleave \cdot \interleave_\mathcal{J}$
(respectively $\interleave \cdot \interleave_\mathcal{J:I}$) a
$p$-additive (respectively, $q$-additive) quasi-norm on
$\mathcal{J}$ (respectively, on $\mathcal{J:I}$) which is
equivalent to the quasi-norm $\|\cdot\|_\mathcal{J}$
(respectively, $\|\cdot\|_\mathcal{J:I}$), where $0<p,q\leqslant
1$. In particular, we have  $\alpha_1\interleave x
\interleave_\mathcal{J}\leqslant\|x\|_\mathcal{J}\leqslant\beta_1\interleave
x \interleave_\mathcal{J}$ and $\alpha_2\interleave a
\interleave_\mathcal{J:I}\leqslant\|a\|_\mathcal{J:I}\leqslant\beta_2\interleave
a \interleave_\mathcal{J:I}$ for all
$x\in\mathcal{J},a\in\mathcal{J:I}$ and some constants
$\alpha_1,\alpha_2,\beta_1,\beta_2>0$. Let
$d_\mathcal{J}(x,y)=\interleave x-y \interleave^p_\mathcal{J},
d_\mathcal{J:I}(a,b)=\interleave a-b \interleave^q_\mathcal{J:I}$
be metrics on $\mathcal{J}$ and $\mathcal{J:I}$ respectively.

Let $\{a_n\}_{n=1}^\infty$ be a Cauchy sequence in $(\mathcal{J:I},d_\mathcal{J:I})$, i.e. $\interleave a_n-a_m \interleave_\mathcal{J:I}^q\leqslant\varepsilon^q$ for all $n,m\geqslant n(\varepsilon)$, thus
\begin{gather}\label{a_nx_fund}
\interleave a_nx-a_mx\interleave_\mathcal{J}\leqslant\frac{1}{\alpha_1}
\|a_nx-a_mx\|_\mathcal{J}\leqslant\frac{1}{\alpha_1}\|a_n-a_m\|_\mathcal{J:I}\|x\|_\mathcal{I} \leqslant
\\\notag
\leqslant\frac{\beta_2}{\alpha_1}\interleave a_n-a_n\interleave_\mathcal{J:I}\|x\|_\mathcal{I}\leqslant\frac{\beta_2}{\alpha_1}
\varepsilon\|x\|_\mathcal{I}
\end{gather}
for all $x\in\mathcal{I}$ for $n,m\geqslant n(\varepsilon)$.
Consequently, the sequence $\{a_nx\}_{n=1}^\infty$ is a Cauchy sequence in $(\mathcal{J},d_\mathcal{J}),x\in\mathcal{I}$.  Since the metric space $(\mathcal{J},d_\mathcal{J})$ is complete, there exists an operator $z(x)\in\mathcal{J}$ such that $\interleave a_nx-z(x) \interleave^p_\mathcal{J}\rightarrow 0$ for $n\rightarrow\infty$. Since $$\|a_nx-z(x)\|_{\mathcal{B}(H)}\leqslant\|a_nx-z(x)\|_\mathcal{J}\leqslant\beta_1\interleave a_nx-z(x)\interleave_\mathcal{J},$$ it follows that $\|a_nx-z(x)\|_{\mathcal{B}(H)}\rightarrow 0$.

Since
$$\|a_n-a_m\|_{\mathcal{B}(H)}\leqslant\|a_n-a_m\|_\mathcal{J:I}\leqslant\beta_2\interleave
a_n-a_m\interleave_\mathcal{J:I}\rightarrow 0$$ for
$n,m\rightarrow\infty$, there exists  $a\in\mathcal{B}(H)$ such
that $\|a_n-a\|_{\mathcal{B}(H)}\rightarrow 0$ for
$n\rightarrow\infty$. For an arbitrary $x\in\mathcal{I}$, we have
$\|a_nx-ax\|_{\mathcal{B}(H)}\leqslant\|a_n-a\|_{\mathcal{B}(H)}\|x\|_\mathcal{I}\rightarrow
0$ for $n\rightarrow\infty$.

Thus, $ax=z(x)$ for all $x\in\mathcal{I}$. Since $z(x)\in\mathcal{J}$ for all $x\in\mathcal{I}$,
it follows that $a\in\mathcal{J:I}$, moreover, due to (\ref{a_nx_fund}),
$\|a_nx-ax\|_\mathcal{J}\leqslant\beta_2\varepsilon\|x\|_\mathcal{I}$ for
$n\geqslant n(\varepsilon)$ and for all $x\in\mathcal{I}$. Consequently,
\begin{gather*}
\interleave a_n-a \interleave_\mathcal{J:I}\leqslant\frac{1}{\alpha_2}\|a_n-a\|_\mathcal{J:I}=\frac{1}{\alpha_2} \sup\bigl\{\|a_nx-ax\|_\mathcal{J}: x\in\mathcal{I}, \|x\|_\mathcal{I}\leqslant 1\bigl\}\leqslant\frac{\beta_2}{\alpha_2}\varepsilon
\end{gather*}
for $n\geqslant n(\varepsilon),$ i.e. $\interleave a_n-a
\interleave_\mathcal{J:I}\rightarrow 0$. Thus, the metric space
$(\mathcal{J:I},d_\mathcal{I:J})$ is complete, i.e.
$(\mathcal{J:I},\|\cdot\|_\mathcal{J:I})$ is a quasi-Banach space.
\end{proof}

\begin{remark}\label{ax=xa} Since the quasi-norms $\|\cdot\|_\mathcal{J}$ and
$\|\cdot\|_\mathcal{J:I}$ are symmetric, for all $a\in\mathcal{J:I}$ the relations
\begin{gather*}
\begin{split}
\|a\|_\mathcal{J:I}&=\|a^*\|_\mathcal{J:I}=
\sup\{\|a^*x\|_\mathcal{J}:x\in\mathcal{I},\|x\|_\mathcal{I}\leqslant 1\}=
\\
&=\sup\{\|x^*a\|_\mathcal{J}:x\in\mathcal{I},\|x\|_\mathcal{I}\leqslant 1\}=
\sup\{\|xa\|_\mathcal{J}:x\in\mathcal{I},\|x\|_\mathcal{I}\leqslant 1\}
\end{split}
\end{gather*}
hold, i.e. for any $a\in\mathcal{J:I}$ we have
\begin{gather}\label{norm a}
\|a\|_\mathcal{J:I}=\sup\{\|xa\|_\mathcal{J}:x\in\mathcal{I},\|x\|_\mathcal{I}\leqslant 1\}.
\end{gather}
\end{remark}

When $\mathcal{I\subseteq J}$ we have $\mathcal{J:I}=\mathcal{B}(H)$ and for any $a\in\mathcal{J:I}$ the mapping $T_a(x)=ax$ is a bounded linear operator from $\mathcal{I}$ to $\mathcal{J}$. As in the proof of Theorem \ref{th_multip} we may establish that $\|a\|_\mathcal{J:I}=\sup\{\|ax\|_\mathcal{J}:x\in\mathcal{I},\|x\|_\mathcal{I}\leqslant 1\}$ is a complete symmetric quasi-norm on $\mathcal{J:I}$. In addition, in case $\mathcal{I=J}$ we have
\begin{gather*}
\begin{split}\|a\|_\mathcal{I:I}&=\sup\{\|ax\|_\mathcal{I}:x\in\mathcal{I},\|x\|_\mathcal{I}\leqslant 1\}\leqslant
\\
&\leqslant\sup\{\|a\|_{\mathcal{B}(H)}\|x\|_\mathcal{I}:x\in\mathcal{I},\|x\|_\mathcal{I}\leqslant 1\}\leqslant\|a\|_{\mathcal{B}(H)},
\end{split}
\end{gather*}
i.e. \begin{gather}\label{I:I}
\|a\|_\mathcal{I:I}\leqslant\|a\|_{\mathcal{B}(H)} \mbox{ for all } a\in\mathcal{I:I}.
\end{gather}
Thus, the norm $\|\cdot\|_{\mathcal{B}(H)}$ and the quasi-norm $\|\cdot\|_\mathcal{I:I}$ are equivalent.

Now, let $G$ and $F$ be arbitrary symmetric quasi-Banach sequence spaces in $l_\infty$. For every $\xi\in F:G$ set
$$\|\xi\|_{F:G}=\sup\{\|\xi\eta\|_F:\eta\in G,\|\eta\|_G\leqslant 1\}.$$

The following theorem may be established similarly to Theorem
\ref{th_multip}.

\begin{theorem}\label{multip_seq}If $G\nsubseteq F$, then $(F:G,\|\cdot\|_{F:G})$ is a symmetric quasi-Banach sequence space in $c_0$ with the modulus of concavity, which does not exceed the modulus of concavity of the quasi-norm $\|\cdot\|_F$, in addition, $\|\xi\eta\|_F\leqslant\|\xi\|_{F:G}\|\eta\|_G$ for all $\xi\in F:G,\eta\in G$.
\end{theorem}

Let $\mathcal{I,J}$ be symmetric quasi-Banach ideals of compact
operators from $\mathcal{B}(H),\mathcal{I}\nsubseteq\mathcal{J}$.
By Proposition \ref{prmultip},
$\mathcal{J:I}=C_{E_\mathcal{J}:E_\mathcal{I}}$, i.e.
$C_{E_\mathcal{J}:E_\mathcal{I}}$ is a two-sided ideal of compact
operators from $\mathcal{B}(H)$. For every $a\in
C_{E_\mathcal{J}:E_\mathcal{I}}$ we set
$$\|a\|_{C_{E_\mathcal{J}:E_\mathcal{I}}}:=\|\{s_n(a)\}\|_{E_\mathcal{I}:E_\mathcal{J}}.$$

\begin{claim}\label{S norm}
$\|\cdot\|_{C_{E_\mathcal{J}:E_\mathcal{I}}}$ is a symmetric quasi-norm on $C_{E_\mathcal{J}:E_\mathcal{I}}$.
\end{claim}
\begin{proof}Obviously, $\|a\|_{C_{E_\mathcal{J}:E_\mathcal{I}}}\geqslant 0$ for all $a\in C_{E_\mathcal{J}:E_\mathcal{I}}$ and $\|a\|_{C_{E_\mathcal{J}:E_\mathcal{I}}}=0 \Leftrightarrow a=0$. If $a,b\in C_{E_\mathcal{J}:E_\mathcal{I}},\lambda\in\mathbb{C}$, then
\begin{gather*}
\|\lambda a\|_{C_{E_\mathcal{J}:E_\mathcal{I}}}=\|\{s_n(\lambda a)\}_{n=1}^\infty\|_{E_\mathcal{J}:E_\mathcal{I}}=
|\lambda|\|a\|_{C_{E_\mathcal{J}:E_\mathcal{I}}}
\end{gather*}
and
\begin{gather*}
\begin{split}
\|a+b\|_{C_{E_\mathcal{J}:E_\mathcal{I}}}&=\|\{s_n(a+b)\}\|_{E_\mathcal{J}:E_\mathcal{I}}
\begin{smallmatrix}
(\ref{sv})\\\leqslant
\\~
\end{smallmatrix}
\|\sigma_2(\{s_n(a)+s_n(b)\})\|_{E_\mathcal{J}:E_\mathcal{I}}\leqslant
\\
&\leqslant 2C\|\{s_n(a)\}+\{s_n(b)\}\|_{E_\mathcal{J}:E_\mathcal{I}}\leqslant
\\
&\leqslant 2C^2(\|\{s_n(a)\}\|_{E_\mathcal{J}:E_\mathcal{I}}+\|\{s_n(b)\}\|_{E_\mathcal{J}:E_\mathcal{I}})=
\\
&=2C^2(\|a\|_{C_{E_\mathcal{J}:E_\mathcal{I}}}+\|b\|_{C_{E_\mathcal{J}:E_\mathcal{I}}}).
\end{split}
\end{gather*}

Hence, $\|\cdot\|_{C_{E_\mathcal{J}:E_\mathcal{I}}}$ is a quasi-norm on $C_{E_\mathcal{J}:E_\mathcal{I}}$ and the modulus of concavity of $\|\cdot\|_{C_{E_\mathcal{J}:E_\mathcal{I}}}$ does not exceed $2C^2$, where $C$ is the modulus of concavity of the quasi-norm $\|\cdot\|_{E_\mathcal{J}}$.

Since
$s_n(xay)\leqslant\|x\|_{\mathcal{B}(H)}\|y\|_{\mathcal{B}(H)}s_n(a)$
for all $a\in\mathcal{K}(H),x,y\in\mathcal{B}(H),n\in\mathbb{N}$
(see Proposition \ref{s-chisla1}), it follows
$$
\|xay\|_{C_{E_\mathcal{J}:E_\mathcal{I}}}=\|\{s_n(xay)\}\|_{E_\mathcal{J}:E_\mathcal{I}}\leqslant
\|x\|_{\mathcal{B}(H)}\|y\|_{\mathcal{B}(H)}\|a\|_{C_{E_\mathcal{J}:E_\mathcal{I}}}.
$$

It is clear that $\|p\|_{C_{E_\mathcal{J}:E_\mathcal{I}}}=1$ for every one-dimensional projection $p$.

Thus, $\|\cdot\|_{C_{E_\mathcal{J}:E_\mathcal{I}}}$ is a symmetric quasi-norm on $C_{E_\mathcal{J}:E_\mathcal{I}}$.
\end{proof}

\begin{remark}
(i) If $\mathcal{I,J}$ are symmetric Banach ideals of compact
operators in $\mathcal{B}(H)$ and
$\mathcal{I}\nsubseteq\mathcal{J}$, then
$(\mathcal{J:I},\|\cdot\|_\mathcal{J:I})$ is a symmetric Banach
ideal of compact operators (Theorem \ref{th_multip}), and
therefore $(E_\mathcal{J:I},\|\cdot\|_{E_\mathcal{J:I}})$ is a
symmetric Banach sequence space in $c_0$ (Theorem \ref{inverse
K-S}).

(ii) If $G,F$ are symmetric Banach sequence spaces in $c_0$ and $G\nsubseteq F$, then $(F:G,\|\cdot\|_{F:G})$ is a symmetric Banach sequence space in $c_0$ (Theorem \ref{multip_seq}), and therefore $(C_{F:G},\|\cdot\|_{C_{F:G}})$ is a symmetric Banach ideal of compact operators from $\mathcal{B}(H)$ (Theorem \ref{K_S}).
\end{remark}

\begin{theorem}\label{th4.5}Let $\mathcal{I,J}$ be symmetric quasi-Banach ideals of compact operators from $\mathcal{B}(H)$ and $\mathcal{I}\nsubseteq\mathcal{J}$. Then

(i) $E_\mathcal{J:I}=E_\mathcal{J}:E_\mathcal{I}$ and $\|\cdot\|_{E_\mathcal{J}:E_\mathcal{I}}\leqslant\|\cdot\|_{E_{\mathcal{J:I}}}\leqslant 2C\|\cdot\|_{E_\mathcal{J}:E_\mathcal{I}}$, where $C$ is the modulus of concavity of the quasi-norm $\|\cdot\|_\mathcal{J}$;

(ii) $\mathcal{J:I}=C_{E_\mathcal{J}:E_\mathcal{I}}$ and $\|\cdot\|_{C_{E_\mathcal{J}:E_\mathcal{I}}}\leqslant\|\cdot\|_\mathcal{J:I}\leqslant 2C\|\cdot\|_{C_{E_\mathcal{J}:E_\mathcal{I}}}$, where $C$ is the modulus of concavity of the quasi-norm $\|\cdot\|_{E_\mathcal{J}}$
\end{theorem}
\begin{proof} If $\xi=\xi^*\in E_\mathcal{J:I}$, then $x_\xi\in\mathcal{J:I}$ (see Theorem \ref{inverse K-S}). Hence, for every $\eta=\eta^*\in E_\mathcal{I}$ we have $x_\eta\in \mathcal{I}$ and $x_{\xi\eta}=x_\xi x_\eta\in\mathcal{J}$, i.e. $\xi\eta\in E_\mathcal{J}$. Therefore, due to Proposition \ref{lem}, $\xi\in E_\mathcal{J}:E_\mathcal{I}$, in addition,
\begin{gather*}
\begin{split}
\|\xi\|_{E_\mathcal{J:I}}&=\|x_\xi\|_\mathcal{J:I}=\sup\{\|x_\xi y\|_\mathcal{J}:y\in\mathcal{I},\|y\|_\mathcal{I}\leqslant 1\}\geqslant
\\
&\geqslant \sup\{\|x_\xi x_\eta\|_\mathcal{J}:\eta\in E_\mathcal{I},\|\eta\|_{E_\mathcal{I}}\leqslant 1\}=
\\
&=\sup\{\|x_{\xi\eta}\|_\mathcal{J}:\eta\in E_\mathcal{I},\|\eta\|_{E_\mathcal{I}}\leqslant 1\}=
\\
&=\sup\{\|\xi\eta\|_{E_\mathcal{J}}:\eta\in E_\mathcal{I},\|\eta\|_{E_\mathcal{I}}\leqslant 1\}=\|\xi\|_{E_\mathcal{J}:E_\mathcal{I}}.
\end{split}
\end{gather*}

Conversely, if $\xi=\xi^*\in E_\mathcal{J}: E_\mathcal{I}$, then $x_\xi\in C_{E_\mathcal{J}:E_\mathcal{I}}=\mathcal{J:I}$ (see Proposition \ref{prmultip}), and so $\xi\in E_\mathcal{J:I}$. Moreover,
\begin{gather*}
\begin{split}
\|\xi\|_{E_\mathcal{J:I}}&=\|x_\xi\|_\mathcal{J:I}=
\sup\{\|x_\xi y\|_\mathcal{J}:y\in\mathcal{I},\|y\|_\mathcal{I}\leqslant 1\}=
\\
&=\sup\{\|x_{\{s_n(x_\xi y)}\}\|_\mathcal{J}:y\in\mathcal{I} ,\|y\|_\mathcal{I}\leqslant 1\}
 \\
&\begin{smallmatrix}(\ref{sv2})\\\leqslant\end{smallmatrix}
\sup\{\|x_{\sigma_2(\{\xi s_n(y)\})}\|_\mathcal{J} :y\in\mathcal{I},\|y\|_\mathcal{I}\leqslant 1\}\leqslant
\\
&\leqslant 2C\sup\{\|\xi\{s_n(y)\}\|_{E_\mathcal{J}}:y\in\mathcal{I},\|y\|_\mathcal{I}\leqslant 1\}\leqslant
\\
&\leqslant
2C\sup\{\|\xi \eta\|_{E_\mathcal{J}}:\eta\in E_\mathcal{I},\|\eta\|_{E_\mathcal{I}}\leqslant 1\}=
2C\|\xi\|_{E_\mathcal{J}:E_\mathcal{I}}.
\end{split}
\end{gather*}

Thus, $E_\mathcal{J:I}=E_\mathcal{J}:E_\mathcal{I}$ and  $\|\xi\|_{E_\mathcal{J}:E_\mathcal{I}}\leqslant\|\xi\|_{E_{\mathcal{J:I}}}\leqslant 2C\|\xi\|_{E_\mathcal{J}:E_\mathcal{I}}$ for all $\xi\in E_\mathcal{J:I}$.

(ii) For an arbitrary $a\in\mathcal{J:I}$ we have
\begin{gather*}
\begin{split}
\|a\|_{C_{E_\mathcal{J}:E_\mathcal{I}}}&=\|\{s_n(a)\}\|_{E_\mathcal{I}:E_\mathcal{J}}=
\\
&=\sup\{\|\{s_n(a)\}\eta\|_{E_\mathcal{J}}:\eta\in E_\mathcal{I},\|\eta\|_{E_\mathcal{I}}\leqslant 1\}=
\\
&=\sup\{\|x_{\{s_n(a)\}}x_\eta\|_\mathcal{J}: x_\eta\in\mathcal{I},\|x_\eta\|_\mathcal{I}\leqslant 1\}\leqslant
\\
&\leqslant
\sup\{\|x_{\{s_n(a)\}}y\|_\mathcal{J}: y\in\mathcal{I},\|y\|_\mathcal{I}\leqslant 1\}=
\\
&=\|x_{\{s_n(a)\}}\|_\mathcal{J:I}=\|a\|_\mathcal{J:I}.
\end{split}
\end{gather*}

On the other hand,
\begin{gather*}
\begin{split}
\|a\|_\mathcal{J:I}&=\sup\{\|ay\|_\mathcal{J}:y\in \mathcal{I},\|y\|_\mathcal{I}\leqslant 1\}=
\\
&=\sup\{\|\{s_n(ay)\}_{n=1}^\infty\|_{E_\mathcal{J}}: y\in\mathcal{I},\|y\|_\mathcal{I}\leqslant 1\}
\\
&\begin{smallmatrix}(\ref{sv2})\\
\leqslant
\\~
\end{smallmatrix}
\sup\{\|\sigma_2(\{s_n(a)s_n(y)\}_{n=1}^\infty)\|_{E_\mathcal{J}}: y\in\mathcal{I},\|y\|_\mathcal{I}\leqslant 1\}=
\\
&=2C\sup\{\|\{s_n(a)s_n(y)\}\|_{E_\mathcal{J}}: y\in\mathcal{I},\|y\|_\mathcal{I}\leqslant 1\}\leqslant
\\
&=2C\|\{s_n(a)\}\|_{E_\mathcal{J}:E_\mathcal{I}}=2C\|a\|_{C_{E_\mathcal{J}:E_\mathcal{I}}}.
\end{split}
\end{gather*}
\end{proof}

Since $(\mathcal{J:I},\|\cdot\|_\mathcal{J:I})$
is a quasi-Banach space (see Theorem
\ref{th_multip}) and quasi-norms $\|\cdot\|_\mathcal{J:I}$ and
$\|\cdot\|_{C_{E_\mathcal{J}:E_\mathcal{I}}}$ are equivalent (see
Theorem \ref{th4.5} (ii)), we have the following corollary:
\begin{corollary}
For any symmetric quasi-Banach ideals $\mathcal{I,J}$ of compact operators from $\mathcal{B}(H),\mathcal{I}\nsubseteq\mathcal{J}$, the couple $(C_{E_\mathcal{J}:E_\mathcal{I}},\|\cdot\|_{C_{E_\mathcal{J}:E_\mathcal{I}}})$ is a symmetric quasi-Banach ideal of compact operators  from $\mathcal{B}(H)$.
\end{corollary}

The following theorem gives the full description of the set $Der(\mathcal{I,J})$.
\begin{theorem}\label{result}
(i) Let $\mathcal{I}$ and $\mathcal{J}$ be symmetric quasi-Banach ideals of compact operators from $\mathcal{B}(H),\mathcal{I}\nsubseteq\mathcal{J}$. Then any derivation $\delta$ from $\mathcal{I}$ to $\mathcal{J}$ has a form $\delta=\delta_a$ for some $a\in C_{E_\mathcal{J}:E_\mathcal{I}}$ and $\|a\|_{\mathcal{B}(H)}\leqslant\|\delta_a\|_\mathcal{B(I,J)}$. Conversely, $\delta_a\in Der(\mathcal{I,J})$ for all $a\in C_{E_\mathcal{J}:E_\mathcal{I}}$. In addition, $\|\delta_a\|_\mathcal{B(I,J)}\leqslant 2C\|a\|_\mathcal{J:I}$, where $C$ is the modulus of concavity of the quasi-norm $\|\cdot\|_\mathcal{J}$;

(ii) Let $G$ and $F$ be symmetric Banach (respectively, $F$ is a $p$-convex, $G$ is a $q$-convex quasi-Banach) sequence spaces in $c_0$ and $G\nsubseteq F$. Then any derivation $\delta\colon C_G\to C_F$ has a form $\delta=\delta_a$ for some $a\in C_{F:G}$ and $\|a\|_{\mathcal{B}(H)}\leqslant\|\delta_a\|_{\mathcal{B}(C_G,C_F)}$. Conversely, $\delta_a\in Der(C_G,C_F)$ for all $a\in C_{F:G}$. In addition, $\|\delta_a\|_{\mathcal{B}(C_G,C_F)}\leqslant 2C\|a\|_{C_F:C_G}$, where $C$ is the modulus of concavity of the quasi-norm $\|\cdot\|_{C_F}.$
\end{theorem}
\begin{proof}(i) According to the Theorem \ref{pr10}, any derivation $\delta\colon \mathcal{I}\to \mathcal{J}$ has a form $\delta=\delta_a$ for some $a\in \mathcal{J:I}$. Since $\mathcal{J:I}=C_{E_\mathcal{J}:E_\mathcal{I}}$ (see Theorem \ref{th4.5}), we have $a\in C_{E_\mathcal{J}:E_\mathcal{I}}$  and $\|a\|_{\mathcal{B}(H)}\leqslant\|\delta_a\|_\mathcal{B(I,J)}$.

Conversely, if $a\in C_{E_\mathcal{J}:E_\mathcal{J}}$, then $a\in \mathcal{J:I}$, and, according to Theorem \ref{pr10}, $\delta_a\in Der(\mathcal{I,J})$.

Moreover,
\begin{gather}\label{norm_der}
\begin{split}
\|\delta_a\|_\mathcal{B(I,J)}&= \sup\{\|\delta_a(x)\|_\mathcal{J}:x\in\mathcal{I},\|x\|_\mathcal{I}\leqslant 1\}=
\\
&=
\sup\{\|ax-xa\|_\mathcal{J}:x\in\mathcal{I},\|x\|_\mathcal{I}\leqslant 1\}\leqslant
\\
&\leqslant
\sup\{C(\|ax\|_\mathcal{J}+\|xa\|_\mathcal{J}):x\in\mathcal{I},\|x\|_\mathcal{I}\leqslant 1\}=
\\
&\begin{smallmatrix}(\ref{norm a})\\=\\~\end{smallmatrix}
2C\sup\{\|ax\|_\mathcal{J}:x\in\mathcal{I},\|x\|_\mathcal{I}\leqslant 1\}=2C\|a\|_\mathcal{J:I}.
\end{split}
\end{gather}

Item (ii) follows from (i) and Theorems \ref{K_S} and \ref{th4.5}. The inequality $\|\delta_a\|_{\mathcal{B}(C_F,C_G)}\leqslant 2C\|a\|_{C_G:C_F}$ is proven in the same manner.
\end{proof}

It is of interest to compare the estimates obtained in Theorem
\ref{result} with the result of L.Zsido \cite{Zsido}, who established that for any derivation $\delta = \delta_a, a\in B(H)$ acting in a von Neumann algebra
 $\mathcal{M}\in \mathcal{B}(H)$ the inequality $\|\delta_a\|_{\mathcal{M}\to\mathcal{M}}\leqslant 2\|a\|_{\mathcal{B}(H)}$ holds. The inequalities established in Theorem \ref{result} for derivations acting in arbitrary symmetric quasi-Banach ideals of compact operators are very similar to those of Zsido. Note, that our proof of these inequalities is based on the technique of $\mathcal{J}$-dual spaces and differs quite markedly from the techniques used in \cite{Zsido}.

We illustrate Theorem \ref{result} with an example drawn from the
theory of Lorentz and Marcinkiewicz sequence spaces. Let
$\omega=\{\omega_n\}_{n=1}^\infty$ be a decreasing weight sequence
of positive numbers. Letting
$W(j)=\sum\limits_{n=1}^jw_n,j\in\mathbb{N}$, we shall assume that
$W(\infty)=\sum\limits_{n=1}^\infty w_n=\infty$.

The Lorentz sequence space $l_\omega^p,1\leqslant p<\infty$, consists of all sequences
$\xi=\{\xi_n\}_{n=1}^\infty\in c_0$ such that $$\|\xi\|_{l_\omega^p}=\biggl(\sum\limits_{n=1}^\infty (\xi_n^*)^pw_n\biggl)^{\frac{1}{p}}<\infty.$$

The Marcinkiewicz sequence space $m_W^p,1\leqslant p<\infty$, is
the space of all sequences $\xi=\{\xi_n\}_{n=1}^\infty\in c_0$
satisfying
$$\|\xi\|_{m_W^p}=\sup\limits_{k\geq 1}\biggl(\frac{\sum\limits_{n=1}^k(\xi^*_n)^p}{W_k}\biggl)^{\frac{1}{p}}<\infty.$$

It is well known (see e.g.\cite{K-P} and \cite{M-P}, Proposition
1) that $(l_\omega^p,\|\cdot\|_{l_\omega^p})$ and
$(m_W^p,\|\cdot\|_{m_W^p})$ are symmetric Banach sequence spaces
in $c_0$.

Hence, $(C_{l_\omega^p},\|\cdot\|_{C_{l_\omega^p}})$ and
$(C_{m_W^p},\|\cdot\|_{C_{m_W^p}})$ are symmetric Banach ideals of
compact operators (Theorem \ref{K_S}). Since  $l_1:
l_\omega=m_W^1$ (see e.g. \cite{K-P}) it follows that $l_p:
l_\omega^p=m_W^p$ for every $1\leqslant p<\infty$ (\cite{M-P}, \S
2). By Theorem \ref{th4.5}, $C_p:C_{l_\omega^p}=C_{m_W^p}$ and
$\|a\|_{C_p:C_{l_\omega^p}}\leqslant 2\|a\|_{C_{m_W^p}}$ for all
$a\in C_p:C_{l_\omega^p}$. From Theorem \ref{result} (ii), we
obtain the following example significantly extending
similar results from \cite{Hoffman}.

\begin{corollary} A linear mapping $\delta\colon C_{l_\omega^p}\to C_p, 1\leqslant p<\infty$
is a derivation if and only if $\delta=\delta_a$  for some $a\in C_{m_W^p}$,
in addition, $\|\delta\|_{\mathcal{B}(C_{l_\omega^p},C_p)}\leqslant 2\|a\|_{C_p:C_{l_\omega^p}}
\leqslant 4\|a\|_{C_{m_W^p}}$.
\end{corollary}

In conclusion, note that, by Theorem \ref{th7} and
(\ref{norm_der}), any derivation
$\delta$ from a symmetric quasi-Banach ideal $\mathcal{I}$ into a symmetric quasi-Banach ideal $\mathcal{J}$, such that $\mathcal{I}\subseteq\mathcal{J}$, has
a form $\delta=\delta_a$ for some $a\in\mathcal{B}(H)$ and, in
addition,
$\|a\|_{\mathcal{B}(H)}\leqslant\|\delta_a\|_\mathcal{B(I,J)}\leqslant
2C\|a\|_{\mathcal{J:I}}$, where $C$ is the modulus of concavity
of the quasi-norm $\|\cdot\|_\mathcal{J}$. Moreover, for the case when $\mathcal{I=J}$ we have $\|a\|_{\mathcal{B}(H)}\leqslant\|\delta_a\|_\mathcal{B(I,J)}\leqslant
2C\|a\|_{\mathcal{B}(H)}$, where $C$ is the modulus of concavity
of the quasi-norm $\|\cdot\|_\mathcal{I}$ (see (\ref{I:I})). This
complements results from \cite{B-S_d}.


\begin{thebibliography}{999}
\bibitem {B-S_d} \quad A.F. Ber, F.A. Sukochev, \textit{Derivations in the Banach ideals of $\tau$-compact operators.}  arXiv:1204.4052 v1. 18
Apr. 2012. 12 p

\bibitem {B-S} \quad A.F. Ber, F.A. Sukochev, \textit{Commutator estimates in $W^*$ algebras.} J. Funct. Anal. \textbf{262} (2012), 537-568.

\bibitem {Calkin} \quad J.W. Calkin, \textit{Two-sided ideals and congruences in the ring of bounded operators in Hilbert space.} Ann. of Math. \textbf{42} (1941), 839-873.

\bibitem {DDdP}\quad P.G. Dodds, T.K. Dodds and B. de Pagter, \textit{
Non-commutative  K\"othe duality}, Trans. Amer. Math. Soc.
 \textbf{339} (1993), 717-750.

\bibitem {F-K} \quad T. Fack, H. Kosaki, \textit{Generalised $s$-numbers of $\tau$-measurable operators}. Pacific. J.Math. \textbf{123} (1986), 265-300.

\bibitem {Gar} \quad D. J. H. Garling, \textit{On ideals of operators in Hilbert space}. Proc. London Math. Soc. \textbf{ 17} (1967), 115-138.

\bibitem {G-K} \quad I. Gohberg, M. Krein, \textit{Introduction to the theory of linear nonselfadjoint operators.} Trans. Math. Monog. \textbf{18} American Mathematical Society. Providence. R.I. 1969.

\bibitem {Hoffman} \quad M.J. Hoffman, \textit{Essential Commutants and Multiplier Ideals.} Indiana Univ. Math. J. \textbf{30} (1981), 859-869.

\bibitem {K-R1} \quad R.Y. Kadison, J.R. Ringrose, \textit{Fundamentals of the Theory of Operator Algerbas I,II}. Academic Press. Orlando. 1983.

\bibitem {Kalton} \quad N.J. Kalton,\textit{Quasi-Banach spaces. Handbook of the geometry of Banach spaces. Vol.2.} Elsevier. (2003), 1099--1130.

\bibitem {K-S} \quad N.J. Kalton, F.A. Sukochev, \textit{Symmetric norms and spaces of operators.} J. Reine Angew. Math.  \textbf{621} (2008), 81--121.

\bibitem{K-P} \quad A.Kami\' nska, A.M.Parrish, \textit{Convexity and concavity constants in Lorentz and Marcinkiewicz spaces.} J.Math.Anal.Appl. \textbf{343} (2008), 337-351.

\bibitem{K-Sh} \quad E.Kissin, V.S.Shulman, \textit{Differential Schatten $*$-algebras. Approximation property and approximate identities.} J.Operator Theory \textbf{45} (2001), 303-334.

\bibitem {K-P-S} \quad  S.G. Krein, Y.I. Petunin, E.M. Semenov, \textit{ Interpolation of linear operators.} Trans. Math. Monogr. \textbf{54.} American Mathematical Society. Providence. R.I. 1982.

\bibitem {M-P} \quad  L.Maligranda, L.E.Persson, \textit{ Generalized duality of some Banach function spaces.} Indag. Math. \textbf{51} (1989), 323-338.

\bibitem {M-Ch} \quad M.A. Muratov, V.I. Chilin,\textit{ Algebras of measurable and locally measurable operators.} Proceedings of the Mathematical Institute. National Academy of Science. Ukraine. \textbf{69}. 2007.

\bibitem {R-S} \quad M. Reed, B. Simon, \textit{Methods of Modern Mathematical Physics. Vol. 1: Functional Analysis.} Academic Press. 1980.

\bibitem{Sak} \quad S. Sakai, \textit{C*-Algebras and W*-Algebras}. Springer-Verlag. New York-Heidelberg-Berlin.  1971.

\bibitem{Schaefer} \quad H.H. Schaefer, \textit{Topological vector spaces}. The Machmillan Company. New York. Collier-Machmillan Limited. London. 1966.

\bibitem{Simon}\quad B. Simon, {\it Trace ideals and their applications,} Second edition. Mathematical Surveys and Monographs, {\bf 120}. American Mathematical Society, Providence, RI, 2005.

\bibitem {V} \quad B.Z. Vulich, \textit{Introduction to the theory of partially ordered spaces.} Wolters-Noordhoff. Groningen. 1967.

\bibitem{Zsido} \quad L. Zsido, \textit{The norm of a derivation in a $W^*$-algebra.} Proc.Amer.Math.Soc. \textbf{38} (1973), 147-150.

\end{thebibliography}
\end{document}